\documentclass[10pt]{article}
\usepackage[english]{babel}
\usepackage{graphicx}
\usepackage{amsmath}
\usepackage{amsfonts}
\usepackage{amssymb}
\usepackage{amsmath}
\usepackage{amssymb}
\usepackage{amsthm}
\newtheorem{lema}{Lemma}[section]
\newtheorem{teor}{Theorem}[section]
\newtheorem{corol}{Corollary}[section]

\newtheorem{prop}{Proposition}[section]
\newtheorem{rem}{Remark}[section]

\newcommand{\nor}[2]{{\left\|{#1}\right\|_{#2}}}
\newcommand{\nora}[3]{{\left\|{#1}\right\|_{#2}^{#3}}}

\title{The IVP for a nonlocal perturbation of the Benjamin-Ono equation in classical and weighted Sobolev spaces}
\author{Germ\'an Fonseca \thanks{Universidad Nacional de Colombia, Bogot\'a. E-mail:  {\tt gefonsecab@unal.edu.co}} \\ Ricardo Pastr\'an \thanks{Universidad Nacional de Colombia, Bogot\'a. E-mail: {\tt rapastranr@unal.edu.co}}\\ Guillermo Rodr\'iguez-Blanco \thanks{Universidad Nacional de Colombia, Bogot\'a. E-mail:  {\tt grodriguezb@unal.edu.co}} 
}
\begin{document}
\maketitle
\begin{abstract}
We prove that the initial value problem associated to a nonlocal perturbation of the Benjamin-Ono equation is locally and globally well-posed in Sobolev spaces $H^s(\mathbb{R})$ for any $s>-3/2$ and we establish that our result is sharp in the sense that the flow map of this equation fails to be $C^2$ in $H^s(\mathbb{R})$ for $s<-3/2$. Finally, we study persistence properties of the solution flow in the weighted Sobolev spaces $Z_{s,r}=H^s(\mathbb{R})\cap L^2(|x|^{2r}\,dx)$ for $s\geq r >0$. We also prove some unique continuation properties of the solution flow in these spaces. 
\end{abstract}
\vspace{0.5cm}
\setcounter{equation}{0}
\setcounter{section}{0}

\section{Introduction and main results}
We study the initial value problem (IVP) for a nonlocal perturbation of the Benjamin-Ono (npBO) equation
\begin{equation}\label{npbo}
\left\{
\begin{aligned}
u_t+uu_x+ \mathcal{H}u_{xx} + \mu (\mathcal{H}u_x+\mathcal{H}u_{xxx})&=0,   \qquad t>0, \; x\in \mathbb{R}, \\
u(0)&=\phi,
\end{aligned}
\right. 
\end{equation}

where $\mu >0$ is constant and $\mathcal{H}$ denotes the usual Hilbert transform given by
\begin{equation*}
\mathcal{H}f(x)=\dfrac{1}{\pi}\,p.v.\int_{-\infty}^{\infty}\dfrac{f(y)}{y-x}\,dy=-\dfrac{1}{\pi}\,v.p.\dfrac{1}{x}\ast f,
\end{equation*}
or equivalently, $\widehat{(\mathcal{H}f)}(\xi)=i\, \text{sgn}(\xi)\widehat{f}(\xi)$ for $f\in \mathcal{S}(\mathbb{R})$. \\ 
This  differential equation corresponds to a nonlocal dissipative perturbation of the Benjamin-Ono equation, npBO. These types of equations have been used in  fluids and plasma theory, see \cite{H} and references therein.
Our aim in this work is to study local and global well-posedness of the initial value problem (IVP) (\ref{npbo}) in classical and weighted Sobolev spaces and to obtain some unique continuation results for the generated  flow. We say that an IVP is locally well-posed (LWP) in the Kato sense in a function spaces $ X$ provided that for every initial data $\phi\in X$ there exist $T=T(\|\phi\|_X)>0$ and a unique solution  $u\in C([0,T]: X)\cap ...=Y_T$ of the given IVP such that the map data-solution is locally continuous from $X$ to $Y_T$, and the IVP is said to be globally well-posed (GWP) in $X$ whenever $T$ can be taken arbitrarily large.  

Well-posedness of the  npBO was first studied by Pastr\'an and Rodr\'iguez in \cite{PR}. They proved that the IVP (\ref{npbo}) is locally well-posed  in $H^s(\mathbb{R})$ for $s>1/2$ and globally well-posed in $H^s(\mathbb{R})$ for $s\geq 1$. In this paper, we show that the initial value problem (\ref{npbo}) is LWP and GWP in the Sobolev spaces $H^s(\mathbb{R})$ for any $s>-3/2$. 
It is interesting to notice that therefore this npBO equation can be solved for more singular initial data than the Benjamin-Ono equation,  obtained from \eqref{npbo} when  the parameter $\mu=0$ for which the largest Sobolev space where it is GWP is $L^2(\mathbb R)$, see \cite{IK}, \cite{MP} and \cite{IT}. The following heuristic scaling argument shows that the Sobolev index $s=-\frac32$ corresponds to the lowest value where well-posedness for IVP (\ref{npbo}) is expected. Given $u$ a solution of the differential equation 
$$
u_t+uu_x+ \mu\mathcal{H}u_{xxx}=0,
$$
with initial data $\phi$ then for every  $\lambda>0, $ $ u_\lambda(x,t)=\lambda^2\, u(\lambda x, \lambda^3\,t)$ is also a solution with initial data $\lambda^2\,\phi(\lambda\cdot)$ and therefore $\|u_\lambda(0)\|_{\dot{H}^s}=\lambda^{2+s-\frac12}\|\phi\|_{\dot{H}^s}$ and hence to have the $\dot{H}^s$ norm invariant under this scaling we should have $s=s_c=-\frac32$. For the Benjamin-Ono equation this scaling index is $s_c=-\frac12$ and as mentioned above well-posedness for the BO equation in the range for $s\in[-\frac12,0)$ is still an open problem.
 
Since the dissipation of the npBO equation is in this  sense ``stronger" than the dispersion, we will use the dissipative methods of Dix for Burgers' equation \cite{Dix},  which consists in applying a fixed point theorem to the integral equation associated to (\ref{npbo}) in a time-weighted space (see (\ref{spacexts}) for the exact definition), see also Pilod \cite{P}, Esfahani \cite{Amin}, Carvajal and Panthee \cite{CP1} and \cite{CP2}, Duque \cite{Duque}, and, Pastr\'an and Ria\~no \cite{PRi}. We also prove that we cannot solve the Cauchy problem by a Picard iterative method implemented on the integral formulation of (\ref{npbo}) for initial data in the Sobolev space $H^s(\mathbb{R})$, $s<-3/2$. In particular, the methods introduced by Bourgain \cite{Bourgain} and Kenig, Ponce and Vega \cite{KPV} for the KdV equation cannot be used for $(\ref{npbo})$ with initial data in the Sobolev space $H^s(\mathbb{R})$ for $s<-3/2$. This kind of ill-posedness result is weaker than the loss of uniqueness proved by Dix in the case of Burgers equation. 

We will mainly work on the integral formulation of the npBO equation,
\begin{equation}\label{intequation}
u(t)=\Psi(u(t)):= S(t)\phi - \int_0^tS(t-\tau)[u(\tau)u_x(\tau)]\,d\tau ,\quad t\geq 0.
\end{equation}
\begin{teor}[LWP]\label{localresult}
Let $\mu >0$ and $s>-3/2$. Then for any $\phi \in H^s(\mathbb{R})$ there exists $T=T(\nor{\phi}{s})>0$ and a unique solution $u$ of the integral equation (\ref{intequation}) satisfying
\begin{align*}
&u\in C([0,T],H^s(\mathbb{R}))\cap C((0,T),H^{\infty}(\mathbb{R})).
\end{align*}
Moreover, the flow map $\phi \mapsto u(t)$ is smooth from $H^s(\mathbb{R})$ to $C([0,T],H^s(\mathbb{R}))\cap C((0,T],H^{\infty}(\mathbb{R}))\cap X_T^s$.
\end{teor}
\begin{teor}[GWP]\label{globalresult}
 Let $s>-3/2$ and $\phi \in H^s(\mathbb{R})$. Then the supremum of all $T>0$ for which all the assertions of Theorem \ref{localresult} hold is infinity. 
\end{teor}

On the other hand, it is known that the Banach's Fixed Point Theorem cannot be applied to the Benjamin-Ono equation  due to the lack of regularity for the map data-solution, more precisely, this map fails to be $C^2$ and even more is not locally uniformly continuous, see  \cite{MolSauTzv} and \cite{KTzv} respectively. 
Here, it is proved that there does not exist a $T>0$ such that (\ref{npbo}) admits a unique local solution defined on the interval $[0,T]$ and such that the flow-map data-solution $\phi \mapsto u(t)$, $t\in [0,T]$, is $C^2$ differentiable at the origin from $H^s(\mathbb{R})$ to $H^s(\mathbb{R})$. As a consequence, we cannot solve the Cauchy problem for the npBO equation by a Picard iterative method implemented on the integral formulation (\ref{intequation}), at least in the Sobolev spaces $H^s(\mathbb{R})$, with $s<-3/2$. This proves that  our  local and global well-posedness results for the npBO in $H^s(\mathbb{R})$, when $s>-3/2$, are  sharp. 

\begin{teor}\label{malpuestodos}
Fix $s<-3/2$. Then there does not exist a $T>0$ such that (\ref{npbo}) admits a unique local solution defined on the interval $[0,T]$ and such that the flow-map data-solution
\begin{equation}
\phi \longmapsto u(t), \qquad t\in [0,T],
\end{equation}
for (\ref{npbo}) is $C^2$ differentiable at zero from $H^s(\mathbb{R})$ to $H^s(\mathbb{R})$.
\end{teor}
A direct corollary of Theorem \ref{malpuestodos} yields our following result:
\begin{teor}\label{illposed}
The flow map data-solution for the npBO equation is not $C^2$ from $H^s(\mathbb{R})$ to $H^s(\mathbb{R})$, if $s<-3/2$.
\end{teor}

On the other hand, we also  study real valued solutions of the IVP npBO (\ref{npbo}) in the weighted Sobolev spaces
\begin{align}
Z_{s,r}=H^s(\mathbb{R})\cap L^2\bigl(|x|^{2r}\,dx\bigr); \quad s,\;r \in \mathbb{R}, \label{wss}
\end{align}
and decay properties of solutions of the IVP npBO (\ref{npbo}). Pastr\'an and Rodr\'iguez in \cite{PR}  proved the following results:
\begin{teor}\label{pr}(See \cite{PR})
Let $\mu>0$ and $T>0$.
\begin{description}
\item[(i)] The \text{IVP} (\ref{npbo}) is \text{GWP} in $Z_{2,1}$.\\
\item[(ii)] If $u(x,t)$ is a solution of the \text{IVP} (\ref{npbo}) such that $u\in C([0,T]:Z_{2,2})$, then $\widehat{u}(0,t)\equiv 0$.\\
\item[(iii)] If $u(x,t)$ is a solution of the IVP (\ref{npbo}) such that $u\in C([0,T]:Z_{3,3})$, then $u(x,t)\equiv 0$.\\
\end{description}
\end{teor} 
Notice that the real valued solutions of the IVP associated to the npBO equation satisfy that the quantity $I(u)=\int_{-\infty}^{\infty}u(x,t)\,dx$ is time invariant, i. e. the property $\widehat{\phi}(0)=0$ is preserved by the solution flow. This leads us to define 
$$\dot{Z}_{s,r}=\{f\in Z_{s,r} \,:\, \widehat{f}(0)=0 \}, \qquad s, r \in \mathbb{R}.$$

 In this work  we extend these results in Theorem \ref{pr} from integer values to the continuum optimal range of indices $(s,r)$. In this sense, our main results are the following:
\begin{teor}\label{pre}
\begin{description}
\item[]
\item[(i)] Let $s\geq r>0$, and $\,r<3/2$.  The \text{IVP} associated to the \text{npBO} equation is \text{GWP}  in $Z_{s,r}$.
\item[(ii)] If  $\,r\in [3/2,5/2)$ and $\,r\leq s$, then the \text{IVP} (\ref{npbo}) is \text{GWP} in $\dot Z_{s,r}$.
\end{description}
\end{teor}
\begin{teor}\label{contunica1}
Let $u\in C([0,T]; Z_{1,1})$ be a solution of the \text{IVP} (\ref{npbo}). If there exist two different times $t_1$, $t_2 \in [0,T]$ such that 
\begin{equation}\label{cont1}
u(\cdot, t_j)\in Z_{3/2,3/2}, \qquad j=1,2, \quad \text{then}\quad \widehat{\phi}(0)=0, \qquad (\text{so}\quad u(\cdot)\in \dot{Z}_{3/2,3/2}).
\end{equation} 
\end{teor}
\begin{teor}\label{contunica2}
Let $u\in C([0,T]; \dot{Z}_{2,2})$ be a solution of the \text{IVP} (\ref{npbo}). If there exist three different times $t_1$, $t_2$, $t_3 \in [0,T]$ such that 
\begin{equation}\label{cont2}
u(\cdot, t_j)\in Z_{5/2,5/2}, \qquad j=1,2,3, \quad \text{then there exists $t^* >t_1$ such that }\quad u(x,t)\equiv 0, \text{for all } \,t\geq t^* .
\end{equation} 
\end{teor}
\begin{rem} It is well-known that Rafael Iorio was the first to establish these type of results. More precisely, his results were obtained  in the context of the famous Benjamin-Ono equation for integer indexes $s,r$, see  \cite{Iorio1}, \cite{Iorio2}. Recently, Fonseca and Ponce, with the help of a carachterization of the classical Sobolev spaces  given by Stein in \cite{S} , extended these results for non-integer values, see \cite{FP}. Fonseca, Linares and Ponce obtained, with the same techniques, similar results for the dispersion generalized Benjamin-Ono equation in \cite{FLP}. For results regarding well-posenedness in these weighted spaces for other dispersive equations as gKdV, Zakharov-Kuznetsov, Benjamin, and Schr\"odinger  see \cite{FLPKDV}, \cite{BJM2} and \cite{FPA}, \cite{J}, \cite{NP},  respectively. 
\end{rem}
\begin{rem}   We note that \textbf{(ii)} and \textbf{(iii)} in Theorem \ref{pr} directly follow as corollaries of Theorems \ref{contunica1} and  \ref{contunica2}, respectively.

\end{rem} 

\subsection{Definitions and Notations}
Given $a$, $b$ positive numbers, $a\lesssim b$ means that there exists a positive constant $C$ such that $a\leq C b$. And we denote $a\sim b$ when, $a \lesssim  b$ and $b \lesssim a$. We will also denote $a\lesssim_{\lambda} b$ or $b\lesssim_{\lambda} a$, if the constant involved depends on some parameter $\lambda$. We will understand $\langle \cdot \rangle = (1+|\cdot|^2)^{1/2}$. We will denote $\widehat{u}(\xi,t)$, $\xi\in\mathbb{R}$, as the Fourier transform of $u(t)$ respect to the variable $x$. We will use the Sobolev spaces $H^s(\mathbb{R})$ equipped with the norm 
$$\nor{\phi}{s}= \nor{\langle \xi \rangle^s\,\widehat{\phi}(\xi)}{L^2(\mathbb{R})},$$
and when $s=0$ we denote the $L^2$ norm simply by $\|\phi\|_0=\|\phi\|$. The norm in the weighted Sobolev spaces is defined by
\begin{equation}\label{sobpeso}
\nora{f}{Z_{s,r}}{2}=\nora{f}{s}{2}+\nora{f}{L_r^2}{2}\, , 
\end{equation}
and $L_r^2(\mathbb{R})=L^2(|x|^{2r}\,dx)$ is the collection of all measurable functions $f:\mathbb{R}\to \mathbb{C}$ such that
\begin{equation}\label{ldospeso}
\nor{f}{L_r^2}=\|\langle x \rangle^r \,f(x)\|<\infty.
\end{equation}
Since the linear symbol of the npBO equation is $b_{\mu}(\xi)= i \xi |\xi| + \mu (|\xi|-|\xi|^3)$, for all $\xi \in \mathbb{R}$, we also denote by $S(t)\phi = e^{t (-\mathcal{H} \partial_x^2 -\mu (\mathcal{H} \partial_x+\mathcal{H} \partial_x^3))}\phi $, for all $t\geq 0$, the semigroup in $H^s(\mathbb{R})$ generated by the operator $-\mathcal{H} \partial_x^2 -\mu (\mathcal{H} \partial_x+\mathcal{H} \partial_x^3)$, i.e.,
\begin{equation}\label{semigroup}
\bigl(S(t)\phi \bigr)^{\wedge}(\xi)=e^{(i\xi|\xi|+\mu(|\xi|-|\xi|^3))t}\widehat{\phi}(\xi)=F_{\mu}(t,\xi)\widehat{\phi}(\xi),
\end{equation}
where $F_{\mu}(t,\xi)=e^{(i\xi|\xi|+\mu(|\xi|-|\xi|^3))t}$, for all $t\geq 0$.\\ \\
Let $0\leq T\leq 1$ and $s < 0$. We consider $X_T^s$ as the class of all the functions $u\in C\left([0,T];H^s(\mathbb{R})\right)$ such that
\begin{equation}\label{spacexts}
\nor{u}{X_T^s}:=\sup_{t\in (0,T]}\Bigl(\nor{u(t)}{s}+t^{|s|/3}\|u(t)\|)\Bigr)<\infty.
\end{equation}
These Banach spaces are an adaptation made by Pilod \cite{P}, of the spaces originally presented by Dix in \cite{Dix}.


\setcounter{equation}{0}
\section{Preliminary estimates}

We first recall some important lemmas which were proved in \cite{PR} and that will also be useful in our arguments
\begin{lema} \label{emuc0enhs} (\cite{PR})
Let $s\in \mathbb{R}$. 
\begin{description}
\item[(i)] $S:[0,\infty)\longrightarrow \textbf{B}(H^s(\mathbb{R}))$ is
a $C^0$-semigroup in $H^s(\mathbb{R})$. Moreover,
\begin{equation}
\nor{S(t)}{\textbf{B}(H^s(\mathbb{R}))}\leq e^{\mu t}. \label{cotaemu}
\end{equation}
\item[(ii)] Let $t>0$ and $\lambda \geq 0$ be given. Then, $S(t)\in \textbf{B}(H^s(\mathbb{R}),H^{s+\lambda}(\mathbb{R}))$ and
\begin{align}
\nor{S(t)\phi}{s+\lambda}\leq C_{\lambda}(e^{\mu t}+(\mu t)^{-\lambda /3})\nor{\phi}{s}\;,  \label{regulariza}
\end{align}
where $C_{\lambda}$ is a constant depending only on $\lambda$.
\item[(iii)] Let $\phi \in H^s(\mathbb{R})$, then $u(t)=S(t)\phi$ is the unique solution of the linear IVP associated to (\ref{npbo}).
\end{description}
\end{lema}
\begin{lema}\label{lemdecaida}
Let $F_{\mu}(t,\xi)=e^{tb_{\mu}(\xi)}$ where $b_{\mu}(\xi)=i\xi|\xi|+\mu(|\xi|-|\xi|^3)$. Then,
\begin{align}
\partial_{\xi}F_{\mu}(t,\xi)&=t[\mu \,sgn(\xi)+|\xi|(2i-3\mu \xi)]F_{\mu}(t,\xi)  \label{uno} \\
\partial_{\xi}^2F_{\mu}(t,\xi)&=2\mu t\delta +t[2i\,sgn(\xi)-6\mu |\xi|]F_{\mu}(t,\xi)+ \notag\\
&+t^2[\mu \,sgn(\xi)+|\xi|(2i-3\mu \xi)]^2F_{\mu}(t,\xi) \label{dos},
\end{align}
where $\delta$ is Dirac's delta distribution.
\end{lema}
\begin{lema} \label{emuc0enfsr}
Let $\mu >0$. $S:[0,+\infty )\longrightarrow \textbf{B}(Z_{s,r})$ is a $C^0$-semigroup for \\ $s,r \in \mathbb{N}$, $s\geq r$ and  satisfies that \\ \\
(a.) If $r=0,1$
\begin{align}
\nor{S(t)\phi}{Z_{s,r}} \leq \Theta_r(t)\nor{\phi}{Z_{s,r}}\quad \text{for all  } \phi \in Z_{s,r} \label{e6}
\end{align}
where $\Theta_r(t)$ has the form
$$p_{\mu ,r}(t)e^{\mu t}+\sum\limits_{l=1}^{3r-1}k_{l,\mu}t^{l/3}$$
such that $k_{l,\mu}$ is a constant which depends on $\mu$ and $p_{\mu ,r}(t)$ is a polynomial in $t$ of degree $r$ with positives coefficients depending only on $\mu $. \newline \newline
b.) If $r\geq 2$ and $\phi \in Z_{s,r}$, $S \in C([0,\infty];Z_{s,r})$, if and only if,
\begin{align}(\partial_{\xi}^j \widehat{\phi})(0)=0, \qquad \qquad j=0,1,2,\cdots,r-2. \label{e7}
\end{align}
In this case, an estimative as (\ref{e6}) holds.
\end{lema}

Regarding our study of the IVP  \eqref {npbo} in the weighted Sobolev spaces, $Z_{s,r}$, we recall the following characterization of the $L_s^p(\mathbb{R})=(1-\Delta)^{-s/2}L^p(\mathbb{R})$ spaces given in \cite{S}.
\begin{teor} \label{derivaStein} Let $b\in (0,1)$ and $2/(1+2b)<p<\infty$. Then $f\in L_b^p(\mathbb{R})$ if and only if
\begin{align}
(a) &f\in L^p(\mathbb{R}), \notag \\
(b) &\mathcal{D}^bf(x)=\Bigl(\int_{\mathbb{R}^n}\frac{|f(x)-f(y)|^2}{|x-y|^{n+2b}}\,dy\Bigr)^{1/2}\in L^p(\mathbb{R}), \label{derivadastein}
\end{align}
with
\begin{equation}\label{normasequivalentes}
\nor{f}{b,p}\equiv \nor{(1-\Delta)^{b/2}f}{L^p}\simeq \nor{f}{L^p}+\nor{D^bf}{L^p}\simeq \nor{f}{L^p}+\nor{\mathcal{D}^bf}{L^p}.
\end{equation}
Above we have used the notation $D^s=(\mathcal{H}\partial_x)^s$ for $s\in \mathbb{R}$. 
\end{teor}
\begin{lema}\label{rprod}(See \cite{FP, NP})
Given $0<b<1$,
\begin{align}
\nor{\mathcal{D}^b(fg)}{}&\leq \nor{g\mathcal{D}^b f }{} + \nor{f \mathcal{D}^bg}{} \label{productostein} \\
\intertext{always that the right side is finite and, for any $t>0$,}
\mathcal{D}^b\Bigl(e^{it\xi |\xi|}\Bigr)&\leq C_b(t^{b/2}+t^b|\xi|^b). \label{cotabo}
\end{align}
\end{lema}
\begin{lema} Let $\mu>0$, $t>0$ and $\lambda \geq 0$. Then,
\begin{equation}\label{unoa}
\nor{|\xi|^{\lambda}e^{\mu t (|\xi|-|\xi|^3)}}{L^{\infty}}\leq C_{\lambda}\bigl(e^{\mu t}+(\mu t)^{-\lambda/3}\bigr)
\end{equation}
\end{lema}
\begin{proof} Let $a>0$ and $\lambda\geq 0$. Since $\xi^{\lambda}e^{a(\xi-\xi^3)}\leq 2^{\lambda/2}e^a$, for $0\leq \xi\leq \sqrt{2}$, and $\xi-\xi^3\leq -\xi^3/2$, for $\xi\geq \sqrt{2}$, we get
\begin{equation}\label{unob}
\sup_{\xi\geq 0} \xi^{\lambda}e^{a(\xi-\xi^3)}\leq c_{\lambda}\Bigl(e^{a} + \sup_{\xi\geq \sqrt{2}}\xi^{\lambda}e^{-a\xi^3/2}\Bigr)
\end{equation}
Taking $g(\xi):= \xi^{\lambda}e^{-a\xi^3/2}$, we note that $g'(\xi)=0$ if and only if $\xi=\sqrt[3]{\frac{2\lambda}{3a}}:=\xi_0$. So, $g(\xi)\leq g(\xi_0)=\bigl(\frac{2\lambda}{3a}\bigr)^{\lambda/3}e^{-\lambda/3}$, for $\xi\geq 0$. Then, from (\ref{unob}), we have that
$$\sup_{\xi\geq 0} \xi^{\lambda}e^{a(\xi-\xi^3)}\leq C_{\lambda}\bigl(e^{a}+a^{-\lambda/3}\bigr)$$
which implies (\ref{unoa}).
\end{proof}
\begin{lema}
Let $b\in (0,1)$ and $h$ a measurable function on $\mathbb{R}$ such that $h$, $h' \in L^{\infty}$. Then,
\begin{equation}\label{dosa}
\mathcal{D}^bh(x)\leq C_b\bigl(\nor{h}{L^{\infty}}+\nor{h'}{L^{\infty}}\bigr)\quad \forall x\in \mathbb{R}.
\end{equation}
\end{lema}
\begin{proof} Given $b\in (0,1)$, we know that
\begin{equation*}
\bigl(\mathcal{D}^bh(x)\bigr)^2=\int_{|x-y|\leq 1}\dfrac{|h(x)-h(y)|^2}{|x-y|^{1+2b}}\,dy + \int_{|x-y|\geq 1}\dfrac{|h(x)-h(y)|^2}{|x-y|^{1+2b}}\,dy.
\end{equation*}
By the Mean Value Theorem $|h(x)-h(y)|\leq \nor{h'}{L^{\infty}}|x-y|$, we obtain that
\begin{align*}
\bigl(\mathcal{D}^bh(x)\bigr)^2&\leq \int_{|x-y|\leq 1}\dfrac{\nora{h'}{L^{\infty}}{2}}{|x-y|^{2b-1}}\,dy + \int_{|x-y|\geq 1}\dfrac{\nora{h}{L^{\infty}}{2}}{|x-y|^{1+2b}}\,dy \\
&\leq c_b\bigl(\nora{h'}{L^{\infty}}{2}+\nora{h}{L^{\infty}}{2}\bigr)
\end{align*}
since $2b-1<1$ and $b>0$. The last inequality implies (\ref{dosa}).
\end{proof}
\begin{corol} Let $b\in (0,1)$. For any $0<t\leq 1$, $0<\mu \leq 1$ and $\lambda \geq 1$, it holds that
\begin{align}
\mathcal{D}^b\Bigl( e^{\mu t(|\xi|-|\xi|^3)}\Bigr)&\leq C_b \bigl(e^{\mu t}+(\mu t)^{1/3}\bigr)\leq C_b, \label{deruno} \\
\mathcal{D}^b\Bigl( e^{\mu t(|\xi|-|\xi|^3)}|\xi|^{\lambda}\Bigr)&\leq C_b \bigl(e^{\mu t}+(\mu t)^{-\lambda/3}\bigr)\leq C_b(\mu t)^{-\lambda/3}. \label{derdos}
\end{align}
\end{corol}
\begin{lema}\label{clave1}
Let $h\in H^b(\mathbb{R})\cap L^2(|x|^{2b})$ where $0<b<1$.  Then, for any $0< t \leq 1$, $0<\mu \leq 1$ and $\lambda \geq 1$,
\begin{align}
\nor{\mathcal{D}^b\bigl(e^{it\xi |\xi|+\mu t(|\xi|-|\xi|^3)}\widehat{h}(\xi)\bigr)}{}
&\leq C_b \Bigl(\nor{\widehat{h}(\xi)}{}+\nor{\,|\xi|^b\widehat{h}(\xi)}{}+\nor{\mathcal{D}^b\bigl(\widehat{h}(\xi)\bigr)}{}\Bigr), \label{dertres} \\
\nor{\mathcal{D}^b\bigl(e^{it\xi |\xi|+\mu t(|\xi|-|\xi|^3)}|\xi|^{\lambda} \widehat{h}(\xi)\bigr)}{}&\leq C_b t^{-\lambda/3}\Bigl(\nor{\widehat{h}(\xi)}{}+\nor{\,|\xi|^b\widehat{h}(\xi)}{}+\nor{\mathcal{D}^b\bigl(\widehat{h}(\xi)\bigr)}{}\Bigr) . \label{dercuatro}
\end{align}
\begin{proof}
To prove (\ref{dertres}) and (\ref{dercuatro}) we use Leibniz's Rule for $\mathcal{D}^b$ (\ref{productostein}) and the results in (\ref{cotabo}), (\ref{unoa}), (\ref{deruno}) and (\ref{derdos}). 
\end{proof}
\end{lema}
\begin{rem}
\begin{description}
\item[]
\item[i.)] (\ref{derdos}) and (\ref{dercuatro}) still hold if $|\xi|^{\lambda}$, for $\lambda \in \mathbb{Z}^{+}$, is substituted by $|\xi|^{\alpha_1}\xi^{\alpha_2}$, where $\lambda=\alpha_1+\alpha_2$, and $\alpha_1$, $\alpha_2 \in \mathbb{Z}^+$.
\item[ii.)] If $0<\lambda <1$, (\ref{derdos}) is not true because the derivative of $|\xi|^{\lambda}$ is not bounded near to zero.
\end{description}
\end{rem}
\begin{lema} Given $\phi \in H^{1+\theta}(\mathbb{R})\cap L^2(|x|^{2(1+\theta)})$ where $\theta \geq 0$ we have that
\begin{align}
\nor{\langle \xi \rangle^{\theta}\partial_{\xi}\widehat{\phi}}{}&\leq C_{\theta} \bigl(\nor{\phi}{1+\theta}+\nor{|x|^{1+\theta}\phi}{}\bigr) \quad \text{and} \label{simplifica} \\
\nor{\langle x \rangle^{\theta}\partial_x\phi}{}&\leq C_{\theta} \bigl(\nor{\phi}{1+\theta}+\nor{|x|^{1+\theta}\phi}{}\bigr). \label{simplificaotra}
\end{align}
\end{lema}
\begin{proof}
Applying the product rule we know that
\begin{align}
\nor{\langle \xi \rangle^{\theta}\partial_{\xi}\widehat{\phi}}{}&\leq \nor{\partial_{\xi}\bigl(\langle \xi \rangle^{\theta}\bigr)\widehat{\phi}}{} + \nor{\partial_{\xi}\bigl(\langle \xi \rangle^{\theta}\widehat{\phi}\bigr)}{} \notag \\
&\leq c_{\theta}\Bigl(\nor{\langle \xi \rangle^{\theta}\widehat{\phi}}{}+\nor{\langle x \rangle J_x^{\theta}\phi}{}\Bigr) \label{simplificauno}
\end{align}
and since 
\begin{align}
\nor{\langle x\rangle J_x^{\theta}\phi}{}&=\nor{\langle x\rangle^{(1+\theta)\frac{1}{1+\theta}} J_x^{(1+\theta)\frac{\theta}{1+\theta}}\phi}{}\leq C\nora{\langle x\rangle^{1+\theta}\phi}{}{\frac{1}{1+\theta}}\nora{J_x^{1+\theta}\phi}{}{\frac{\theta}{1+\theta}}\notag \\
&\leq C\bigl(\nor{\langle x\rangle^{1+\theta}\phi}{}+\nor{\phi}{1+\theta}\bigr), \label{simplificados}
\end{align} 
then, using (\ref{simplificados}) in (\ref{simplificauno}), we obtain the inequality (\ref{simplifica}). The proof of (\ref{simplificaotra}) is similar.
\end{proof}
As a further direct consequence of Theorem \ref{derivaStein} we will use the following result in the proof of Theorem \ref{contunica1}, deduced in \cite{FP}.
\begin{prop}\label{I1}
Let $p\in (1, \infty)$. If $f\in L^p(\mathbb{R})$ such that there exists $x_0 \in \mathbb{R}$ for which $f(x_0+)$, $f(x_0-)$ are defined and $f(x_0+)\ne f(x_0-)$, then for any $\delta > 0$, $\mathcal{D}^{1/p}f \notin L_{loc}^p(B(x_0,\delta))$ and consequently $f\notin L_{1/p}^p(\mathbb{R})$.
\end{prop}
Also, we will employ the next simple estimate.
\begin{prop}\label{simplestimate}
If $f\in L^2(\mathbb{R})$ and $\phi \in H^1(\mathbb{R})$, then
$$\nor{[D^{1/2}, \phi] \,f}{}\leq c\,\nor{\phi}{1}\,\nor{f}{}.$$
\end{prop}


\setcounter{equation}{0}
\section{Theory in $H^s(\mathbb{R})$}
The purpose in this section is to prove LWP and GWP of the IVP (\ref{npbo}) in Sobolev spaces $H^s(\mathbb{R})$ for $s> -3/2$. Our strategy is to use a contraction argument on the integral equation (\ref{intequation}) associated to (\ref{npbo}). We have introduced in (\ref{spacexts}) the $X_T^s$ spaces, for $0\leq T\leq 1$ and $s < 0$, in order to obtain linear and bilinear estimates. First, we recall the following lemma, in \cite{Amin}, which is useful in establishing smoothness properties for the semigroup $S$ of (\ref{npbo}).
\begin{lema}\label{LemmaLB1}
Let $\lambda>0$ and $0<t\leq 9 \lambda$ be given. Then
\begin{equation}\label{LBequa1}
\xi^{2\lambda} e^{t \left(|\xi|-|\xi|^3 \right)}\leq f_{\lambda}(t):=\rho^{2\lambda}\,e^{t(\rho -\rho^3)},
\end{equation}
where
$$\rho=\dfrac{\left(9\lambda +\sqrt{81\lambda^2-t^2}\,\right)^{1/3}}{3}\,t^{-1/3}+\dfrac{t^{1/3}}{3\left(9\lambda +\sqrt{81\lambda^2-t^2}\,\right)^{1/3}}.$$
Moreover, if $\lambda=0$, then (\ref{LBequa1}) holds for $f_0(t)=\exp \left( \frac{2t}{3\sqrt{3}}\right)$.
\end{lema}
\begin{proof}See \cite{Amin}.
\end{proof}
Now, we are going to estimate the linear part of (\ref{intequation}) in $X_T^s$.
\begin{prop}\label{PropLB1}
Let $0<T\leq T^*= \min \{1, 9|s|/2\}$, $s<0$ and $\phi\in H^s(\mathbb{R})$, then it follows that
\begin{equation}\label{LB1a}
\sup_{t\in[0,T]}\left\|S(t)\phi\right\|_{s}\leq e^{\frac{2}{3\sqrt{3}}T} \left\|\phi\right\|_{s},
\end{equation}
and
\begin{equation}\label{LB1b}
\sup_{t\in[0,T]}t^{\frac{|s|}{3}}\left\|S(t)\phi\right\|\lesssim_{s}g_{s}(T)\left\|\phi\right\|_{s},
\end{equation}
where 
$$g_{s}(t)=e^{\frac{2 t}{3\sqrt{3}}}+ t^{\frac{|s|}{3}}\,f_{|s|/2}(t),$$
is a continuous nondecreasing function on $[0,T^*]$ and $f$ is defined as in Lemma (\ref{LemmaLB1}).
\end{prop}
\begin{proof}
It is the same proof of Proposition 1 in \cite{Amin}.
\end{proof}

Next, we establish the crucial bilinear estimates. 
\begin{prop}\label{PropLB2} Let $0\leq t\leq T\leq T^*$ and $-\frac{3}{2}<s < 0$, then
\begin{equation}
\left\|\int_{0}^t S(t-t')\partial_x(uv)(t') \ dt'\right\|_{X_T^s} \lesssim_{s} e^{\frac{2\sqrt{2}\,\mu T}{\sqrt{27}}} T^{\frac{2s+3}{6}}\left\|u\right\|_{X_T^s}\left\|v\right\|_{X_T^s},
\end{equation}
for all $u,v\in X_T^s$.
\end{prop}
\begin{proof}
Since $s< 0$, it follows that $\left\langle\xi \right\rangle^{s}\leq |\xi|^s$, for all real number $\xi$ different from zero. Then we deduce that
\begin{equation}\label{LBequa2}
\begin{aligned}
&  \left\|\int_{0}^t S(t-t')\partial_x(uv)(t') \ dt'\right\|_{s} \\
& \hspace{30pt} \leq  \int_{0}^t \left\|\left\langle \xi \right\rangle^s e^{\mu \left(|\xi|-|\xi|^3\right)(t-t')} \left(\partial_x(uv)(t')\right)^{\wedge}(\xi) \right\|_{} \ dt' \\
& \hspace{30pt} \leq  \int_{0}^t \left\||\xi|^{1+s}e^{\mu\left(|\xi|-|\xi|^3\right)(t-t')}\right\|\left\| \widehat{u(t')}\ast \widehat{v(t')}(\xi) \right\|_{L^{\infty}(\mathbb{R})} \ dt'. 
\end{aligned}
\end{equation}
The Young inequality implies that
\begin{equation}\label{LBequa3}
\left\| \widehat{u(t')}\ast\widehat{v(t')}(\xi) \right\|_{L^{\infty}(\mathbb{R})}\leq \left( \frac{\left\|u\right\|_{X_T^s}\left\|v\right\|_{X_T^s}}{|t'|^{2|s|/3}} \right),
\end{equation}
thus we obtain
\begin{equation}\label{LBequa4}
\begin{aligned}
\int_{0}^t & \left\| S(t-t')\partial_x(uv)(t') \right\|_{s}  \ dt'  \\
& \qquad \leq  \int_{0}^t \frac{\left\||\xi|^{1+s}e^{\mu\left(|\xi|-|\xi|^3\right)t'}\right\|_{}}{|t-t'|^{2|s|/3}} \, dt' \, \left\|u\right\|_{X_T^s}\left\|v\right\|_{X_T^s}. 
\end{aligned}
\end{equation}
To estimate the integral on the right-hand side of \eqref{LBequa4}, we perform the change of variables $w=t^{1/3}\xi$ to deduce 
\begin{align}\label{LBequa5}
\left\||\xi|^{1+s}e^{\mu\left(|\xi|-|\xi|^3\right)t}\right\|& \leq \frac{\left\||w|^{1+s}e^{-\frac{\mu |w|^3}{2}}\right\|\left\|e^{\mu( |w|t^{2/3}- \frac{ |w|^3}{2})}\right\|_{L^{\infty}(\mathbb{R})}}{|t|^{\frac{2s+3}{6}}} \nonumber \\
& \lesssim_{s} \frac{e^{\frac{2\sqrt{2} \,\mu  T}{\sqrt{27}}}}{|t|^{\frac{2s+3}{6}}},
\end{align}
where we have used the following inequality
$$e^{\mu( |\xi|t^{2/3}- \frac{ |\xi|^3}{2})}\leq e^{\frac{2\sqrt{2}\,\mu t}{\sqrt{27}}}, \qquad \text{for all}\;\,\xi \in \mathbb{R}.$$
Therefore, we get from \eqref{LBequa4} and \eqref{LBequa5} that
\begin{equation}\label{LBequa6}
\begin{aligned}
& \left\|\int_{0}^t  S(t-t')\partial_x(uv)(t')  \ dt' \right\|_{s}  \\
& \hspace{20pt} \lesssim_s e^{\frac{2\sqrt{2}\, \mu T}{\sqrt{27}}}T^{\frac{1}{6}(3+2s)}\left( \int_{0}^1 \frac{1}{|\sigma|^{\frac{2|s|}{3}}|1-\sigma|^{\frac{2s+3}{6}}}\ d\sigma \right)\left\|u\right\|_{X_T^s}\left\|v\right\|_{X_T^s},   
\end{aligned}
\end{equation}
for all $0\leq t \le T$. On the other hand, arguing as above, we have for all $0\leq t\leq T$ that
\begin{align}\label{LBequa7}
 t^{|s|/3}&\left\|\int_{0}^t S(t-t')\partial_x(uv)(t') \ dt'\right\|_{} \nonumber \\
&\leq t^{|s|/3} \int_{0}^t \left\||\xi|e^{\mu \left(|\xi|-|\xi|^3\right)(t-t')}\right\|\left\| \widehat{u(t')}\ast \widehat{v(t')}(\xi) \right\|_{L^{\infty}(\mathbb{R})} \ dt' \nonumber \\
&  \leq t^{|s|/3}\int_{0}^t \frac{\left\||\xi|e^{\mu \left(|\xi|-|\xi|^3\right)t'}\right\|}{|t-t'|^{2|s|/3}} \, dt' \, \left\|u\right\|_{X_T^s}\left\|v\right\|_{X_T^s}  \nonumber \\
& \lesssim_s e^{\frac{2\sqrt{2}\,\mu T}{\sqrt{27}}} t^{|s|/3} \left( \int_{0}^t |t'|^{-\frac{1}{2}}|t-t'|^{-2|s|/3}\, dt' \right)\left\|u\right\|_{X_T^s}\left\|v\right\|_{X_T^s} \nonumber \\
& \lesssim_s e^{\frac{2\sqrt{2}\,\mu T}{\sqrt{27}}} T^{\frac{1}{6}(3+2s)} \left( \int_{0}^1 |\sigma|^{-2|s|/3}|1-\sigma|^{-1/2}\, d\sigma \right)\left\|u\right\|_{X_T^s}\left\|v\right\|_{X_T^s}.
\end{align}
Combing \eqref{LBequa6} and \eqref{LBequa7} the proof is complete.
\end{proof}
\begin{rem}
If we consider $s'>s>-\frac{3}{2}$, then modifying the space $X_{T}^{s'}$ by
$$\tilde{X}_{T}^{s'}=\left\{u\in X_{T}^{s'}: \left\|u\right\|_{\tilde{X}_{T}^{s'}}<\infty \right\},$$
where 
$$ \left\|u\right\|_{\tilde{X}_{T}^{s'}}= \left\|u\right\|_{X_{T}^{s'}}+t^{|s|/3} \left\|(1-\partial_x^2)^{\frac{ s'-s}{2}} u\right\|$$
and using that
$$(1+\xi^2)^{s'/2}\lesssim (1+\xi^2)^{s/2}(1+\xi_1^2)^{(s'-s)/2}+(1+\xi^2)^{s/2}\left(1+(\xi-\xi_1)^2\right)^{(s'-s)/2},$$
for all $\xi,\xi_1\in \mathbb{R}$, we deduce arguing as in Proposition \ref{PropLB2} that
$$\left\|\int_{0}^t S(t-t')\partial_x(uv)(t') \ dt'\right\|_{\tilde{X}_T^{s'}} \lesssim_{s}e^{\frac{2\sqrt{2}\,\mu T}{\sqrt{27}}} T^{\frac{2s+3}{6}}\left(\left\|u\right\|_{\tilde{X}_T^{s'}}\left\|v\right\|_{X_T^s}+\left\|u\right\|_{X_T^s}\left\|v\right\|_{\tilde{X}_T^{s'}}\right).$$
\end{rem}
Next regularization property is a consequence of the semi-group property in Lemma \ref{emuc0enhs} $(ii)$,  and we refer to  Proposition 4 in \cite{P} for its proof. 
\begin{prop}\label{PropLB3}
Let $0\leq T \leq 1$, $s \in (-\frac{3}{2},0)$ and $\delta\in [0,s+\frac{3}{2})$, then the application 
$$t\longmapsto \int_{0}^t S(t-t')\partial_x(u^2)(t')\ dt' ,$$
is in $C\left((0,T];H^{s+\delta}(\mathbb{R})\right)$, for every $u\in X_T^s$.
\end{prop}

\subsection{LWP in $H^s(\mathbb{R})$ for $s\in (-3/2,0)$}
\begin{proof}[Proof of Theorem \ref{localresult}] We divide the proof in four steps
\\ \\
1. \emph{Existence}. Let $\phi \in H^s(\mathbb{R})$ with $s> -\frac{3}{2}$. We consider the application
$$\Psi(u)=S(t)\phi-\frac{1}{2}\int_{0}^t S(t-t')\partial_x(u^2(t')) \ dt',$$
for each $u \in X_T^s$. By Proposition \ref{PropLB1} together with Proposition \ref{PropLB2}, when $s< 0$, there exists a positive constant $C=C(\mu,s)$ such that 
\begin{align}
\left\|\Psi(u)\right\|_{X_T^s} &\leq C\left(\left\|\phi\right\|_{s}+T^{g(s)}\left\|u\right\|_{X_T^s}^2\right), \label{WPequa1} \\
\left\|\Psi(u)-\Psi(v)\right\|_{X_T^s} & \leq C T^{g(s)}\left\|u-v\right\|_{X_T^s}\left\|u+v\right\|_{X_T^s}, \label{WPequa2}
\end{align}
for all $u, v \in X_T^s$ and $0<T\leq 1$. Where $g(s)=\frac{1}{6}(3+2s)$, when $s\in (-\frac{3}{2},0)$. Then, we define $E_{T}(\gamma)=\left\{u\in X_T^s : \left\|u \right\|_{X_T^s}\leq\gamma \right\}$, with $\gamma=2C\left\|\phi \right\|_{s}$ and $0<T\leq \min \left\{1,\left(4C\gamma\right)^{-\frac{1}{g(s)}} \right\}$. The estimates \eqref{WPequa1} and \eqref{WPequa2} imply that $\Psi$ is a contraction on the complete metric space $E_T(\gamma)$.  Therefore, the Fixed Point Theorem implies the existence of a unique solution $u$ of \eqref{intequation} in $E_T(\gamma)$ with $u(0)=\phi$. 
\\ \\
2. \emph{Continuous dependence}. We will verify that the map $\phi \in H^s(\mathbb{R}) \mapsto u \in X_T^s$, where $u$ is a solution of \eqref{npbo} obtained in the step of \emph{Existence} is continuous. More precisely, for $s>-\frac{3}{2}$, if $\phi_n \rightarrow \phi_{\infty}$ in $H^s(\mathbb{R})$, let $u_n\in X_{T_n}^s$ be the respective solutions of \eqref{intequation} (obtained in the part of  \emph{Existence}) with $u_n(0)=\phi_n$, for all $1\leq n\leq \infty$. Then for each $T'\in (0,T_{\infty})$, $u_n \in X_{T'}^s$ (for $n$ large enough) and $u_n \rightarrow u_{\infty}$ in $X_{T'}^s$.

We recall that the solutions and times of existence previously constructed satisfy 
\begin{align}
&0<T_n\leq \min\left\{1, \left(8C^2\left\|\phi_n\right\|_{s}\right)^{-\frac{1}{g(s)}}\right\}, \label{WPequa3} \\
&\left\|u_n\right\|_{X_T^s} \leq 2C\left\|\phi_n\right\|_{s}, \label{WPequa4}
\end{align}
for all $n\in \mathbb{N}\cup\left\{\infty\right\}$. Let $T'\in (0,T_{\infty})$, the above inequalities and the hypothesis imply that there exists $N\in \mathbb{N}$, such that for all $n\geq N$, we have that $T'\leq T_n$ and
$$\frac{\left\|\phi_n\right\|_{s}+\left\|\phi_{\infty}\right\|_{s}}{\left\|\phi_{\infty}\right\|_{s}}\leq 3.$$
Therefore, combining \eqref{WPequa3}, \eqref{WPequa4} with the Propositions \ref{PropLB1} and \ref{PropLB2}, it follows that for each $n\geq N$ 
\begin{align*}
\left\|u_n-u_{\infty}\right\|_{X_{T'}^s} &\leq C\left\|\phi_n-\phi_{\infty}\right\|_{s}+ CT_{\infty}^{g(s)}\left\|u_n+u_{\infty}\right\|_{X_{T'}^s}\left\|u_n-u_{\infty}\right\|_{X_{T'}^s} \\
&\leq C\left\|\phi_n-\phi_{\infty}\right\|_{s}+ \frac{\left(\left\|\phi_n\right\|_{s}+\left\|\phi_{\infty}\right\|_{s}\right)}{4\left\|\phi_{\infty}\right\|_{s}}\left\|u_n-u_{\infty}\right\|_{X_{T'}^s} \\
&\leq C\left\|\phi_n-\phi_{\infty}\right\|_{s}+ \frac{3}{4}\left\|u_n-u_{\infty}\right\|_{X_{T'}^s}. 
\end{align*}
Hence we have deduced that $\left\|u_n-u_{\infty}\right\|_{X_{T'}^s}\leq C\left\|\phi_n-\phi_{\infty}\right\|_{s}$, for all $n\geq N$.
\\ \\
3. \emph{Uniqueness}. Let $u,v \in X_T^s$ be solutions of the integral equation \eqref{intequation} on $[0,T]$ with the same initial data. For each $r\in [0,T]$ we define
$$
G_{r}(t)=
\begin{cases}
\frac{1}{2}\int_{r}^{t} S(t-t')\left(\partial_x u^2(t')-\partial_x v^2(t')\right)dt', & \text{if }t\in (r,T] \\
0, & \text{if }t\in[0,r]
\end{cases}
$$
for all $t \in [0,T]$. Arguing as in the proof of Proposition \ref{PropLB2} we deduce that there exists a positive constant $C=C(\mu,s)$ depending only on $\mu$ and $s$, such that for all $r\in[0,T]$ and all $\vartheta\in [r,T]$,
\begin{equation}\label{WPequa5}
\left\|G_{r}\right\|_{X_{\vartheta}^s} \leq C K \left(\vartheta-r\right)^{g(s)}\left\|u-v\right\|_{X_{\vartheta}^s},
\end{equation}
where $K=\left\|u\right\|_{X_{T}^s}+\left\|v\right\|_{X_{T}^s}$. In particular, inequality \eqref{WPequa5} implies that
\begin{equation}\label{WPequa6}
\left\|u-v\right\|_{X_{\vartheta}^s}=\left\|G_{0}\right\|_{X_{\vartheta}^s} \leq C K \vartheta^{g(s)}\left\|u-v\right\|_{X_{\vartheta}^s}.
\end{equation}
Thus, choosing $\vartheta \in \left(0,(CK)^{-\frac{1}{g(s)}}\right)$ a fixed number, \eqref{WPequa6} implies that $u \equiv v$ on $[0,\vartheta]$. Therefore we can iterate this argument using \eqref{WPequa5} and our choose of $\vartheta$, until we extend the uniqueness result to the whole interval $[0,T]$.
\\ \\
 \emph{4. The solution $ u\in C\left((0,T],H^{\infty}(\mathbb{R})\right)$}. From Lemma \ref{LemmaLB1} and arguing as in the proof of Proposition 2.2 in \cite{BI}, we have that the map $t\mapsto S(t)\phi$ is continuous in the interval $(0,T]$ with respect to the topology of $H^{\infty}(\mathbb{R})$. Since our solution $u$ is in $X_T^s$, we deduce from Proposition \ref{PropLB3} that, there exists $\lambda>0$ such that 
$$u\in C\left([0,T];H^s(\mathbb{R})\right)\cap C\left((0,T];H^{s+\lambda}(\mathbb{R})\right).$$ 
Therefore we can iterate this argument, using uniqueness result and the fact that the time of existence of solutions depends uniquely on the $H^s(\mathbb{R})$-norm of the initial data. Thus we deduce that 
$$u\in C\left([0,T];H^s(\mathbb{R})\right)\cap C\left((0,T];H^{\infty}(\mathbb{R})\right).$$
\end{proof}
\subsection{LWP in $H^s(\mathbb{R})$ for $s\geq 0$}
For simplicity, we assume that  $\mu=1$ and $0<T\leq 1$. We will mainly work with the integral formulation (\ref{intequation}) of the IVP (\ref{npbo}).
\begin{lema} Let $\mu=1$, $0\leq s\leq 1/2$, $0\leq \tau \leq t \leq T\leq 1$ and $u\in C([0,T], H^s(\mathbb{R}))$.
\begin{equation}\label{cotast}
\int_0^t\nor{S(t-\tau)\partial_xu^2(\tau)}{s}\,d\tau \leq C_s\,T^{(3-2s)/6}\,\nora{u}{L_T^{\infty}H_x^s}{2}
\end{equation}
\end{lema}
\begin{proof}
\begin{align}
&\nora{S(t-\tau)\partial_xu^2(\tau)}{s}{2}=\int_{\mathbb{R}}(1+\xi^2)^se^{2(|\xi|-|\xi|^3)(t-\tau)}\xi^2|\widehat{u}\ast \widehat{u}(\xi)|^2\,d\xi \notag \\
&\qquad \leq c_s  \Bigl(\int_{\mathbb{R}}\xi^2e^{2(|\xi|-|\xi|^3)(t-\tau)},d\xi + \int_{\mathbb{R}}\xi^{2(s+1)}e^{2(|\xi|-|\xi|^3)(t-\tau)},d\xi \Bigr)\,\nora{\widehat{u}\ast \widehat{u}(\xi)}{L_{\xi}^{\infty}}{2} \label{cotastuno}
\end{align}
Since $\xi-\xi^3\leq 1$, for $0\leq \xi\leq \sqrt{2}$, and $\xi-\xi^3\leq -\xi^3/2$, for $\xi\geq \sqrt{2}$, we have
\begin{align}
\int_0^{\infty}\xi^2e^{2(\xi-\xi^3)t},d\xi &\leq \int_0^{\sqrt{2}}\xi^2e^{2 t}\,d\xi + \int_{\sqrt{2}}^{\infty}\xi^2e^{- t \xi^3}\,d\xi \leq c(e^{2 t}+(3 t)^{-1}), \label{cotastcuatro} \\
\intertext{and}
\int_0^{\infty}\xi^{2(s+1)}e^{2(\xi-\xi^3)t},d\xi &\leq \int_0^{\sqrt{2}}\xi^{2s+2}e^{2 t}\,d\xi + \int_{\sqrt{2}}^{\infty}\xi^{2s+2}e^{- t \xi^3}\,d\xi \notag \\
&\leq c_s \Bigl(e^{2 t}+\dfrac{\Gamma(2s+1)}{3( t)^{1+2s/3}}\Bigr). \label{cotastcinco}
\end{align}
Then, from (\ref{cotastuno}), (\ref{cotastcuatro}), (\ref{cotastcinco}) and Young's inequality
\begin{equation}\label{cotastdos}
\nor{S(t-\tau)\partial_xu^2(\tau)}{s} \leq c_s \Bigl(e^{(t-\tau)}+\dfrac{1}{(t-\tau)^{1/2}}+\dfrac{1}{(t-\tau)^{(2s+3)/6}}\Bigr)\,\nora{u(\tau)}{}{2}.
\end{equation}
Integrating from $0$ to $t$ we obtain
\begin{equation}\label{cotasttres}
\int_0^t\nor{S(t-\tau)\partial_xu^2(\tau)}{s}\,d\tau \leq C_s \Bigl(e^t-1+2t^{1/2}+\dfrac{6}{3-2s}t^{(3-2s)/6}\Bigr)\,\nora{u}{L_t^{\infty}H_x^s}{2}.
\end{equation}
So, we can conclude (\ref{cotast}).
\end{proof}
\begin{proof}[Proof of Theorem \ref{localresult}]  For $T\in(0,1]$, we consider the space $X_T^s=C\left([0,T];H^s(\mathbb{R})\right)$. Let $\phi \in H^s(\mathbb{\mathbb{R}})$, $0\leq s\leq 1/2$. We define the application
\begin{equation}\label{intaplication}
\Psi(u)=S(t)\phi-\frac{1}{2}\int_{0}^t S(t-\tau)\partial_x(u^2(\tau)) \ d\tau, \text{ for each } u \in X_T^s.
\end{equation}
By (\ref{cotaemu}) and (\ref{cotast}), there exists a positive constant $C_s$, such that for all $u, v \in X_T^s$ and $0<T\leq 1$
\begin{align}
\left\|\Psi(u)\right\|_{X_T^s} &\leq C\left(\left\|\phi\right\|_s+T^{\frac{3-2s}{6}}\left\|u\right\|_{X_T^s}^2\right), \label{WP1} \\
\left\|\Psi(u)-\Psi(v)\right\|_{X_T^s} & \leq C T^{\frac{3-2s}{6}}\left\|u-v\right\|_{X_T^s}\left\|u+v\right\|_{X_T^s}, \label{WP2}
\end{align}
for all $s\in [0,\frac{1}{2}]$. Then, let $E_{T}(a)=\left\{u\in X_T^s : \left\|u \right\|_{X_T^s} \leq a=2C\left\|\phi \right\|_s \right\}$, where 
\begin{equation*}
2CaT^{\frac{-2s+3}{6}}\leq \frac{1}{2}\quad \text{ i.e.}\quad 0<T \leq \min \left\{1,\left(4Ca\right)^{\frac{6}{2s-3}} \right\}.
\end{equation*} 
The estimates \eqref{WP1} and \eqref{WP2} imply that $\Psi$ is a contraction on the complete metric space $E_T(a)$. Therefore, we deduce by the Fixed Point Theorem that there exists an unique solution $u$ of the integral equation \eqref{intequation} in $E_T(a)$ and with initial data $u(0)=\phi$. Furthermore, the existence time satisfies 
\begin{equation}\label{timeexistence}
T\lesssim \nora{\phi}{s}{6/(2s-3)}.
\end{equation}
The rest of the proof follows canonical arguments, so we omit it.
\end{proof}
\begin{rem}
From the inequality of regularization (\ref{regulariza}) for the semigroup $S(t)$ and a Gronwall's type inequality (see 1.2.1 in \cite{H}) we have that for the solutions of the IVP (\ref{npbo}), $u(t)\in H^{\infty}(\mathbb{R})$ for all $t>0$ and, in particular, for all $t\in (0,T]$ and $0\leq s\leq1/2$,
\begin{equation}\label{estimativaunmedio+}
\nor{u(t)}{1/2+}\leq C \nor{\phi}{s} t^{(s-(\frac{1}{2}+))/3}.
\end{equation}
\end{rem}
\begin{rem}
When $s>1/2$, $H^s(\mathbb{R})$ is a Banach algebra and the local theory for the IVP (\ref{npbo}) is reduced to consider the space $X_T^s=C\left([0,T];H^s(\mathbb{T})\right)$ and $E_{T}(a)=\left\{u\in X_T^s : \left\|u \right\|_{X_T^s} \leq a=2C\left\|\phi \right\|_s \right\}$ where $\phi \in H^s(\mathbb{R})$ and $T$ will be chosen. We define the application $\Psi(u)$ as in (\ref{intaplication}) and by (\ref{cotaemu}) and (\ref{regulariza}) we have easily that
\begin{align}
\nor{\Psi(u)}{s} &\leq \nor{S(t)\phi}{s}+1/2\int_0^t\nor{S(t-\tau)u^2(\tau)}{s+1}\,d\tau \notag \\
&\leq C\Bigl(\nor{\phi}{s}+ \int_0^t\dfrac{\nor{u^2}{s}}{(t-\tau)^{1/3}}\,d\tau\Bigr) \notag \\
&\leq C\bigl(\nor{\phi}{s}+ \nora{u}{X_T^s}{2}T^{2/3}\bigr)\leq \dfrac{a}{2}+Ca^2T^{2/3}. \label{lwps>1/2}
\end{align}
Hence, according to (\ref{lwps>1/2}) we choose $T$ such that $T^{2/3}< \frac{1}{4C^2\nor{\phi}{s}}$ to obtain that $\Psi$ is a contraction. So, IVP (\ref{npbo}) is LWP in $H^s(\mathbb{R})$ for $s>1/2$ and the existence time of the solution satisfies
\begin{equation}\label{times>1/2}
T\sim \nora{\phi}{s}{-3/2}.
\end{equation} 
\end{rem}
In a similar way to Proposition \ref{PropLB3}, we have  
\begin{prop}\label{regularity}
Let $0\leq T \leq 1$, $s\geq 0$ and $\delta\in [0,s+\frac{3}{2})$, then the application 
$$t\longmapsto \int_{0}^t S(t-t')\partial_x(u^2)(t')\ dt' ,$$
is in $C\left((0,T];H^{s+\delta}(\mathbb{R})\right)$, for every $u\in C([0,T];H^s(\mathbb{R}))$.
\end{prop}

\subsection{GWP in $H^s(\mathbb{R})$ for $s> -3/2$}

\begin{proof}[Proof of Theorem \ref{globalresult}] Let  $s\geq 0$ and $\phi \in H^s(\mathbb{R})$. It is known that $S(\cdot)\phi$ belongs to $C([0,\infty),H^s(\mathbb{R}))\cap C((0,\infty), H^{\infty}(\mathbb{R}))$. By the Proposition \ref{regularity} we have that
\begin{equation*}
t\longmapsto \int_0^tS(t-t')\,\partial_x(u^2(t'))\,dt'\;\in C([0,T],H^{s+2\delta}(\mathbb{R})),
\end{equation*}
where $u \in C([0,T];H^s(\mathbb{R}))$ is the solution to (\ref{intequation}) that we have already got. So we conclude that
\begin{equation*}
u\in C([0,T],H^s(\mathbb{R}))\cap C((0,T],H^{s+2\delta}(\mathbb{R})).
\end{equation*}
From above we can deduce by induction that $u\in C((0,T],H^{\infty}(\mathbb{R}))$. Define $T^*=T^*(\nor{\phi}{s})$ by
\begin{equation}
T^*=\sup \bigl\{T>0: \exists ! \;\;\text{solution of (\ref{intequation}) in }C([0,T],H^s(\mathbb{R})) \bigr\}.
\end{equation}
Let $u\in C([0,T^*),H^s(\mathbb{R}))\cap C((0,T^*),H^{\infty}(\mathbb{R}))$ be the local solution of (\ref{intequation}) in the maximal time interval $[0,T^*)$. We shall prove that if we assume $T^*<\infty$, then a contradiction follows. Since $u$ is smooth, we deduce that $u$ solves the Cauchy problem (\ref{intequation}) in classical sense, which allows us to take the $L^2$ scalar product of (\ref{intequation}) with $u$ and integrate by parts to obtain
\begin{align*}
\dfrac{1}{2}\dfrac{d}{dt}\|u(t)\|^2&=(u,u_t)_0 \notag \\
&=-(u,uu_x)_0 -(u,\mathcal{H} u_{xx})_0 -\mu(u,\mathcal{H}u_x)_0-\mu(u,\mathcal{H}u_{xxx})_0  \notag \\
&= \mu \int_{\mathbb{R}}(|\xi|-|\xi|^3)|\Hat{u}(\xi)|^2\,d\xi \notag \\
&= \mu \Bigl(\int_{|\xi|\leq 1}(|\xi|-|\xi|^3)|\Hat{u}(\xi)|^2\,d\xi + \int_{|\xi|>1}(|\xi|-|\xi|^3)|\Hat{u}(\xi)|^2\,d\xi\Bigr) \label{partiendo1} \\
&\leq \mu \int_{|\xi|\leq 1}(|\xi|-\xi^2)|\Hat{u}(\xi)|^2\,d\xi \notag \\
&\leq \mu \int_{|\xi|\leq 1}|\Hat{u}(\xi)|^2\,d\xi \notag \\
&\leq \mu \|u(t)\|^2.
\end{align*}
Integrating the last relation between $0$ and $t$, it gives
\begin{align}
\|u(t)\|^2\leq \|\phi\|^2& + 2\mu \int_0^t \|u(\tau)\|^2\,d\tau . \notag
\end{align}
Using the Gronwall's inequality we obtain a priori estimate
\begin{equation}
\|u(t)\|\leq \|\phi\|\,e^{\mu T^*}\equiv M, \quad \forall t\in (0,T^*). \notag
\end{equation}
Since the time existence $T(\cdot)$ is a decreasing function of the norm of the initial data, we know that there exists a time $T_1>0$ such that for all $\varphi \in L^2(\mathbb{R})$, with $\nor{\varphi}{L^2}\leq M$, there exists a unique solution $v(x,t)$ of (\ref{intequation}) satisfying $v(0)=\varphi$ and $v\in C([0,T_1],L^2(\mathbb{R}))\cap C((0,T_1],H^{\infty}(\mathbb{R}))$. Now, we choose $0<\epsilon <T_1$, apply this result with $\varphi=u(T^*-\epsilon)$ and define
\begin{equation*}
\tilde{u}(t)=\left\{
\begin{aligned}
u(t),\qquad \qquad \qquad &\text{when  }\; 0\leq t\leq T^*-\epsilon, \\
v(t-(T^*-\epsilon)), \quad &\text{when  }\;T^*-\epsilon \leq t\leq T^*-\epsilon+T_1.
\end{aligned}
\right.
\end{equation*}
Then $\tilde{u}$ is a solution of (\ref{intequation}) in the time interval $[0,T^*-\epsilon +T_1]$, which contradicts $T^*<\infty$, since $T^*-\epsilon +T_1>T^*$. This implies that the solution can be extended to infinite time. \\ \\
Now, let $s\in (-3/2,0)$, $\phi\in H^s(\mathbb{R})$ and $u\in X_{T}^s$ be the solution of the integral equation \eqref{intequation}, obtained in above steps, and  let $T'\in (0,T)$ fixed. We have that
$$
\left\|u\right\|_{X_{T'}^s}=M_{T',s}<\infty.
$$
Since $u\in C\left((0,T];H^{\infty}(\mathbb{R})\right)$, it follows that $u(T')\in L^2(\mathbb{R})$. Thus, the GWP result in $H^s$ for $s\geq 0$ implies that $\tilde{u}$, the solution of \eqref{intequation} with initial data $u(T')$, is global in time. Moreover, uniqueness implies that $\tilde{u}(t)=u(T'+t)$ for all $t\in [0,T-T']$. Therefore, we deduce that
\begin{align*}
\left\|u\right\|_{X_{T}^s} & \leq \left\|u\right\|_{X_{T'}^s}+\left\|u(T'+\cdot)\right\|_{X_{T-T'}^s} \\
&\leq M_{T',s}+\left\|\tilde{u}\right\|_{X_{T-T'}^s} \\
&= M_{T',s}+ \sup_{t\in [0,T-T']} \left\{\left\|\tilde{u}(t)\right\|_{s}+t^{|s|/3}\left\|\tilde{u}(t)\right\|\right\} \\
& \leq M_{T',s}+ \left(1+(T-T') ^{|s|/3}\right)\sup_{t\in [0,T-T']}\left\|\tilde{u}(t)\right\|. 
\end{align*}
The global result follows from the above estimate.
\end{proof}


\subsection{Ill-posedness type results}
In this section we prove the ill-posedness result contained in Theorem \ref{malpuestodos}.
\begin{teor}\label{malpuestouno}
Let $s<-\frac{3}{2}$ and $T>0$. Then there does not exist a space $X_T$ continuously embedded in $C([-T,T],H^s(\mathbb{R}))$ such that there exists $C>0$ with
\begin{align}
\nor{S(t)\phi}{X_T}&\leq C\,\nor{\phi}{s}; \qquad \phi \in H^s(\mathbb{R}), \label{illone} \\
\intertext{and}
\nor{\int_0^tS(t-t')[u(t')u_x(t')]\,dt'}{X_T}&\leq C \nora{u}{X_T}{2}; \qquad u\in X_T. \label{illtwo}
\end{align}
\end{teor} 
Note that (\ref{illone}) and (\ref{illtwo}) would be needed to implement a Picard iterative scheme on (\ref{intequation}), in the space $X_T$. 

\begin{proof}[Proof of Theorem \ref{malpuestouno}]
Suppose that there exists a space $X_T$ such that (\ref{illone}) and (\ref{illtwo}) hold. Take $u=S(t)\phi$ in (\ref{illtwo}). Then
\begin{equation}
\nor{\int_0^tS(t-t')[(S(t')\phi)(S(t')\phi_x)]\,dt'}{X_T}\leq C\,\nora{S(t)\phi}{X_T}{2}.
\end{equation}
Now using (\ref{illone}) and that $X_T$ is continuously embedded in $C([-T,T],H^s(\mathbb{R}))$ we obtain for any $t\in [-T,T]$ that
\begin{equation}
\nor{\int_0^tS(t-t')[(S(t')\phi)(S(t')\phi_x)]\,dt'}{s}\leq C\,\nora{\phi}{s}{2}. \label{illthree}
\end{equation}
We show that (\ref{illthree}) fails by choosin an appropriate $\phi$. Take $\phi$ defined by its Fourier transform as
\begin{equation}
\widehat{\phi}(\xi)=N^{-s}\,\gamma^{-1/2}\,(\mathbb{I}_I(\xi) + \mathbb{I}_I(-\xi)) \label{funcionfi} 
\end{equation}
where $I$ is the interval $[N,N+2\gamma]$ and $\gamma \ll N$. Note that $\nor{\phi}{s}\sim 1$. Taking $p(\xi)=\mu(|\xi|-|\xi|^3)$ and $q(\xi)=\xi |\xi|$, we have that
\begin{align}
\int_0^t&S(t-t')[(S(t')\phi)(S(t')\phi_x)]\,dt' \notag \\
&=\int_0^t \int_{\mathbb{R}} e^{i x \xi}F_{\mu}(t-t',\xi)(i\xi)\Bigl[F_{\mu}(t',\cdot)\widehat{\phi} \ast F_{\mu}(t',\cdot)\widehat{\phi}\Bigr](\xi)\,d\xi \,dt' \notag \\
&= i \int_{\mathbb{R}^2}e^{i x \xi +t(p(\xi)+i q(\xi))}\xi \; \widehat{\phi}(\xi - \xi_1)\,\widehat{\phi}(\xi_1) \int_0^te^{t'[p(\xi-\xi_1)+p(\xi_1)-p(\xi)+i(q(\xi-\xi_1)+q(\xi_1)-q(\xi))]}\,dt' \,d\xi_1 \,d\xi \notag \\ 
&= i \int_{\mathbb{R}^2}e^{i x \xi +t(p(\xi)+i q(\xi))}\xi \; \widehat{\phi}(\xi - \xi_1)\,\widehat{\phi}(\xi_1) \int_0^te^{t'[\chi(\xi,\xi_1)+i\psi(\xi,\xi_1)]}\,dt' \,d\xi_1 \,d\xi \label{cuentadesiempre}
\end{align}
where 
\begin{align*}
\chi(\xi,\xi_1)&=p(\xi-\xi_1)+p(\xi_1)-p(\xi)=\mu (|\xi-\xi_1|-|\xi-\xi_1|^3+|\xi_1|-|\xi_1|^3-|\xi|+|\xi|^3) \\
\intertext{and}
\psi(\xi,\xi_1)&=q(\xi-\xi_1)+q(\xi_1)-q(\xi)=(\xi-\xi_1)|\xi-\xi_1|+\xi_1|\xi_1|-\xi|\xi| .
\end{align*}
Since $$\widehat{\phi}(\xi - \xi_1)\,\widehat{\phi}(\xi_1)=N^{-2s}\gamma^{-1}\Bigl[\mathbb{I}_I(\xi-\xi_1)\,\mathbb{I}_I(-\xi_1)+\mathbb{I}_I(-(\xi-\xi_1))\,\mathbb{I}_I(\xi_1)\Bigr]$$
we define
$$K_{\xi}:=\{\xi_1:\xi_1\in I, \xi-\xi_1\in -I\}\cup \{\xi_1:\xi_1\in -I, \xi-\xi_1\in I\}$$
then
\begin{align}
\Bigl(\int_0^t&S(t-t')[(S(t')\phi)(S(t')\phi_x)]\,dt' \Bigr)^{\wedge}(\xi) \notag \\
&= i\,\xi\,e^{t(p(\xi)+i q(\xi))}\int_{\mathbb{R}} \widehat{\phi}(\xi - \xi_1)\,\widehat{\phi}(\xi_1) \int_0^te^{t'[\chi(\xi,\xi_1)+i\psi(\xi,\xi_1)]}\,dt' \,d\xi_1 \notag \\
&=i\,\xi\,e^{t(p(\xi)+i q(\xi))}\int_{K_{\xi}} N^{-4s}\,\gamma^{-2} \int_0^te^{t'[\chi(\xi,\xi_1)+i\psi(\xi,\xi_1)]}\,dt' \,d\xi_1 .\label{fourierdenolineal}
\end{align}
We thus deduce that
\begin{align}
&\nora{\int_0^tS(t-t')[(S(t')\phi)(S(t')\phi_x)]\,dt'}{s}{2} \notag \\
&\geq \int_{-2\gamma}^{2\gamma}(1+\xi^2)^s|\xi|^2e^{2tp(\xi)}N^{-4s}\,\gamma^{-2}\Bigl|\int_{K_{\xi}}  \int_0^te^{t'[\chi(\xi,\xi_1)+i\psi(\xi,\xi_1)]}\,dt' \,d\xi_1\Bigr|^2\,d\xi \notag \\
&=\int_{-2\gamma}^{2\gamma}(1+\xi^2)^s|\xi|^2e^{2tp(\xi)}N^{-4s}\,\gamma^{-2}\Bigl|\int_{K_{\xi}}  \dfrac{e^{t'[\chi(\xi,\xi_1)+i\psi(\xi,\xi_1)]}-1}{\chi(\xi,\xi_1)+i\psi(\xi,\xi_1)} \,d\xi_1\Bigr|^2\,d\xi \notag \\
&\geq \int_{-2\gamma}^{2\gamma}(1+\xi^2)^s|\xi|^2e^{2tp(\xi)}N^{-4s}\,\gamma^{-2}\Bigl(\int_{K_{\xi}}  \Re\Bigl(\dfrac{e^{t'[\chi(\xi,\xi_1)+i\psi(\xi,\xi_1)]}-1}{\chi(\xi,\xi_1)+i\psi(\xi,\xi_1)} \Bigr) \,d\xi_1\Bigr)^2\,d\xi. \label{importante}
\end{align}
Since, 
\begin{align*}
\chi(\xi,\xi_1)\sim -\mu N^3 \quad\text{and}\quad |\psi(\xi,\xi_1)| &\sim \gamma N \qquad \text{for all}\qquad \xi_1\in K_{\xi} \\
\intertext{then}
(e^{t\chi}\cos(t\psi)-1)\chi &\gtrsim -\mu N^3\,e^{-\mu N^3 t} \\
\psi \sin(t\psi)\,e^{t\chi} \geq -|\psi|e^{t\chi} &\gtrsim -\gamma \,N\,e^{-\mu N^3 t}   \\
\intertext{and so,}
\chi^2+\psi^2 &\sim N^2(\mu^2\,N^4 + \gamma^2).
\end{align*}
Hence,
\begin{equation}
\Bigl(\int_{K_{\xi}}  \Re\Bigl(\dfrac{e^{t'[\chi(\xi,\xi_1)+i\psi(\xi,\xi_1)]}-1}{\chi(\xi,\xi_1)+i\psi(\xi,\xi_1)} \Bigr) \,d\xi_1\Bigr)^2 \gtrsim \gamma^2\;\dfrac{e^{-2\mu N^3 t}}{N^2(\mu N^2+\gamma)^2} . \label{importanteuno}
\end{equation}
Therefore, from (\ref{importante}) and (\ref{importanteuno})
\begin{align}
\nora{\int_0^tS(t-t')[(S(t')\phi)(S(t')\phi_x)]\,dt'}{s}{2}&\gtrsim \int_{-2\gamma}^{2\gamma} (1+\gamma^2)^s\,\gamma^2\,N^{-4s}\,\dfrac{e^{-2\mu N^3t}}{N^2(\mu N^2+\gamma)^2}\,d\xi \notag \\
&\sim (1+\gamma^2)^s\,\gamma^3\,N^{-4s-2}\,\dfrac{e^{-2\mu N^3t}}{(\mu N^2 + \gamma)^2}. \label{importantedos}
\end{align}
Taking $\gamma=O(1)$ it infers for $N\gg \gamma$ and any $T>0$ that
\begin{equation*}
\sup_{t\in [0,T]}\nor{\int_0^tS(t-t')[(S(t')\phi)(S(t')\phi_x)]\,dt'}{s} \gtrsim N^{-2s-3}.
\end{equation*}
This contradicts (\ref{illthree}) for $N$ large enough, since $\nor{\phi}{s}\sim 1$ and $-2s-3>0$ when $s<-3/2$.
\end{proof}
As a consequence of Theorem \ref{malpuestouno} we can obtain the following result.
\begin{proof}[Proof of Theorem \ref{malpuestodos}]
Consider the Cauchy problem
\begin{equation}\label{ostbourgain}
\left\{
\begin{aligned}
u_t+u_{xxx}+\mu (\mathcal{H}u_x + \mathcal{H}u_{xxx})+uu_x&=0, \\
u(x,0)&=\alpha \phi(x), \quad \alpha\ll 1, \quad \phi\in H^s(\mathbb{R}). 
\end{aligned}
\right.
\end{equation}
Suppose that $u(\alpha,x, t)$ is a local solution of (\ref{ostbourgain}) and that the flow map is $C^2$ at the origin from $H^s(\mathbb{R})$ to $H^s(\mathbb{R})$. We have 
\begin{equation*}
\dfrac{\partial^2u}{\partial \alpha^2}(0,x,t)=-2\int_0^tS(t-t')[(S(t')\phi)(S(t')\phi_x)]\,dt'.
\end{equation*}
The assumption of $C^2$ regularity yields 
\begin{equation*}
\sup_{t\in [0,T]}\nor{-2\int_0^tS(t-t')[(S(t')\phi)(S(t')\phi_x)]\,dt'}{s}\leq C\,\nora{\phi}{s}{2},
\end{equation*}
but this is exactly the estimate which has been shown to fail in the proof of Theorem \ref{malpuestouno}.
\end{proof}


\setcounter{equation}{0}
\section{Theory in $Z_{s,r}$ for $s\geq r>0$}

\begin{prop}\label{xbuenl2}
Let $b\in (0,1/2]$, $s>-3/2$ and $u$ be the solution of the integral equation (\ref{intequation}) with initial data $\phi \in H^s(\mathbb{R})$. If $|x|^b\phi \in L^2(\mathbb{R})$ then $|x|^bu(t)\in L^2(\mathbb{R})$ for all $t\in [0,T]$.
\end{prop}
\begin{proof}
We employ the integral equation (\ref{intequation}) and so, for all $t\in [0,T]$,
\begin{align}
\nor{|x|^bu(t)}{}\leq \nor{|x|^bS(t)\phi}{} + \int_0^t\nor{|x|^bS(t-\tau)[u(\tau)u_x(\tau)]}{}\,d\tau. \label{xalabinte}
\end{align} 
Now, using the Fourier transform, Stein's derivative $\mathcal{D}^b$ and applying (\ref{dercuatro}) and (3.12), from \cite{LP}, we have that
\begin{align}
\nor{|x|^bS(t-\tau)\partial_xu^2(\tau)}{}&\simeq \nor{\mathcal{D}^b\bigl(e^{i(t-\tau )\xi |\xi|+\mu (t-\tau )(|\xi|-|\xi|^3)}\xi \widehat{u^2}(\xi,\tau)\bigr)} \notag \\
&\leq C_b \frac{1}{(t-\tau)^{1/3}}\bigl( \nor{u^2(\tau)}{b}+\nor{|x|^bu^2(\tau)}{}\bigr) \label{preone} \\
&\leq C_b \frac{1}{(t-\tau)^{1/3}}\bigl( \nor{u(\tau)}{L^{\infty}}\nor{u(\tau)}{b}+\nor{u(\tau)}{L^{\infty}}\nor{|x|^bu(\tau)}{}\bigr) \notag \\
&\leq C_b\bigl(\nor{u}{X_T^b}+\nor{|x|^bu}{L_t^{\infty}H_x^0}\bigr)\, \frac{\nor{u(\tau)}{\frac{1}{2}+}}{(t-\tau)^{1/3}}. \label{one}
\end{align}
Since for given $\alpha$, $\beta \in [0,1)$ it holds that
\begin{equation*}
\int_0^t\dfrac{d\tau}{(t-\tau)^{\alpha}\,\,\tau^{\beta}}\leq c_{\alpha, \beta}\, t^{1-\alpha -\beta} .
\end{equation*}
This estimate combined with (\ref{estimativaunmedio+}) give us
\begin{align}
\int_0^t\nor{|x|^bS(t-\tau)[u(\tau)u_x(\tau)]}{}\,d\tau \leq C_{b,s}\bigl(\nor{u}{X_T^b}+\nor{|x|^bu}{L_t^{\infty}H_x^0}\bigr)\,\nor{\phi}{s}\,T^{\frac{2}{3}-\frac{(1/2+)-s}{3}}. \label{two}
\end{align}
Hence, by (\ref{xalabinte}), (\ref{dertres}) and (\ref{two}), it follows that
\begin{equation}\label{three}
\sup_{t\in [0,T]}\nor{|x|^bu(t)}{}\leq C_b\bigl( \nor{\phi}{b}+\nor{|x|^b\phi}{}\bigr)+C_{b,s}\bigl(\nor{u}{X_T^b}+\nor{|x|^bu}{L_t^{\infty}H_x^0}\bigr)\,\nor{\phi}{s}\,T^{\frac{2}{3}-\frac{(1/2+)-s}{3}}.
\end{equation}
By Theorem \ref{localresult} we know that $T$ depends explicitly on $\nor{\phi}{s}=a/2C$ and since $0<T\leq 1$ we see that
\begin{equation*}
T^{\frac{2}{3}-\frac{(1/2+)-s}{3}}\leq T^{\frac{-2s+3}{6}}\; \Longleftrightarrow \;\frac{-2s+3}{6}\leq \frac{2}{3}-\frac{(1/2+)-s}{3}\; \Longleftrightarrow \; s\geq 0+.
\end{equation*}
Then, taking $s=b$, we obtain from (\ref{three}) that
\begin{align}
\sup_{t\in [0,T]}\nor{|x|^bu(t)}{}&\leq C_b\bigl( \nor{\phi}{b}+\nor{|x|^b\phi}{}\bigr)+C_{b}\,a \nor{\phi}{b}\,T^{\frac{-2b+3}{6}}+C_{b}\nor{\phi}{b}\,T^{\frac{-2b+3}{6}}\nor{|x|^bu}{L_t^{\infty}H_x^0} \notag \\
&\leq C_b\bigl( \nor{\phi}{b}+\nor{|x|^b\phi}{}\bigr)+C_{b} \nor{\phi}{b}+\dfrac{1}{4}\nor{|x|^bu}{L_t^{\infty}H_x^0}. \label{four}
\end{align}
So,
\begin{equation}\label{five}
\sup_{t\in [0,T]}\nor{|x|^bu(t)}{}\leq C_b\bigl( \nor{\phi}{b}+\nor{|x|^b\phi}{}\bigr).
\end{equation}
\end{proof}
\begin{rem}
The proof of the Proposition \ref{xbuenl2} shows us that the solution of the IVP (\ref{npbo}) persists in $L^2(|x|^{2b}dx)$ for the same time of existence $T=T(\nor{\phi}{b})$ when $0< b\leq 1/2$.
\end{rem}
\begin{prop}\label{xbuenl21/2<b<1}
Let $b\in (1/2,1)$ and $u$ be the solution of the integral equation (\ref{intequation}) with initial data $\phi \in H^b(\mathbb{R})$. If $|x|^b\phi \in L^2(\mathbb{R})$ then $|x|^bu(t)\in L^2(\mathbb{R})$ for all $t\in [0,T]$.
\end{prop}
\begin{proof}
The proof is similar to that of Proposition \ref{xbuenl2} because the inequalities (\ref{dertres}) and (\ref{dercuatro}) still valid when $b\in (1/2,1)$. However, within this range $H^b(\mathbb{R})$ is a Banach algebra, therefore from inequality (\ref{preone}) we have that
\begin{align}
\nor{|x|^bS(t-\tau)\partial_xu^2(\tau)}{}&\leq C_b \frac{1}{(t-\tau)^{1/3}}\bigl( \nora{u(\tau)}{b}{2}+\nor{|x|^bu(\tau)}{}\nor{u(\tau)}{L^{\infty}}\bigr)
\intertext{and}
\int_0^t\nor{|x|^bS(t-\tau)\partial_xu^2(\tau)}{}\,d\tau &\leq C_b\bigl(\nora{u}{X_T^b}{2}+ \nor{u}{X_T^b}\nor{|x|^bu}{L_t^{\infty}H_x^0}\bigr)\,T^{2/3} \notag \\
&\leq \bigl(4C^3\nora{\phi}{b}{2}+2C^2\nor{\phi}{b}\nor{|x|^bu}{L_t^{\infty}H_x^0}\bigr)\,T^{2/3} \notag \\
&\leq C\nor{\phi}{b} + \frac{1}{2}\nor{|x|^bu}{L_t^{\infty}H_x^0}. \label{pretwo}
\end{align}
In the last inequality, it was used the choice of $T$ in (\ref{times>1/2}). So, we can conclude (\ref{five}) with $1/2<b<1$.
\end{proof}
\begin{lema}
Let $\theta \in (0,1/2)$, $b=1+\theta$ and $\phi \in Z_{b,b}$. Then, for any $0< t \leq 1$
\begin{align}
\nor{|x|^{1+\theta}S(t)\phi}{}&\leq C_b\,t\,\bigl(\nor{|x|^{1+\theta}\phi}{}+\nor{\phi}{1+\theta}\bigr)\leq C_b \bigl(\nor{|x|^b\phi}{}+\nor{\phi}{b}\bigr) \label{first} \\
\nor{|x|^{1+\theta}S(t)\partial_x\phi}{}&\leq C_b\,t^{-1/3}\,\bigl(\nor{|x|^b\phi}{}+\nor{\phi}{b}\bigr) \label{second}
\end{align}
\end{lema}
\begin{proof}
We will denote $F_{\mu}(t,\xi)$ when $\mu=1$ simply by $F(t,\xi)$. So, applying (\ref{uno}), we have that
\begin{align}
&\nor{|x|^{1+\theta}S(t)\phi}{}=\nor{|x|^{\theta}xS(t)\phi}{}\leq \nor{|x|^{\theta}S(t)(x\phi)}{}+\nor{t|x|^{\theta}S(t)\bigl(\mathcal{H}+2D_x-3D_x\partial_x \bigr)\phi}{}. \label{firsttwo}
\end{align}
From (\ref{dertres}), substituting $\widehat{\phi}$ by $\partial_{\xi}\widehat{\phi}$, and applying (\ref{simplifica}), we find that
\begin{align}
\nor{\mathcal{D}^{\theta} \bigl(F(t,\xi)\partial_{\xi}\widehat{\phi}(\xi)\bigr)}{}&\leq C_{\theta} \Bigl(\nor{\partial_{\xi}\widehat{\phi}(\xi)}{}+\nor{\,|\xi|^{\theta}\partial_{\xi}\widehat{\phi}(\xi)}{}+\nor{\mathcal{D}^{\theta}\bigl(\partial_{\xi}\widehat{\phi}(\xi)\bigr)}{}\Bigr) \notag \\
&\leq C_{\theta} \Bigl(\nor{x\phi}{}+\nor{\langle \xi \rangle^{\theta}\partial_{\xi}\widehat{\phi}(\xi)}{}+\nor{|x|^{\theta}x\phi}{} \Bigr) \notag \\
&\leq C_{\theta} \bigl(\nor{|x|^{1+\theta}\phi}{}+ \nor{\phi}{1+\theta} \bigr). \label{firstthree}
\end{align}
It follows from (\ref{firstthree}) that the first term on the right hand side of (\ref{firsttwo}) is bounded by
\begin{align}
\nor{|x|^{\theta}S(t)(x\phi)}{}&\leq C\Bigl(\nor{S(t)(x\phi)}{}+\nor{\mathcal{D}^{\theta} \bigl(F(t,\xi)\partial_{\xi}\widehat{\phi}(\xi)\bigr)}{}\Bigr) \notag \\
&\leq C\bigl(\nor{\phi}{b}+\nor{|x|^b\phi}{}\bigr). \label{firstfive}
\end{align}
The second term of the right hand side of (\ref{firsttwo}) is bounded by 
\begin{align}
& \nor{tS(t)\bigl(\mathcal{H}-3D_x\partial_x +2D_x\bigr)\phi}{}+\nor{t\,\mathcal{D}^{\theta} \bigl(F(t,\xi) sgn(\xi) \widehat{\phi}(\xi)\bigr)}{} + \nor{t\,\mathcal{D}^{\theta} \bigl(F(t,\xi) |\xi| \widehat{\phi}(\xi)\bigr)}{} \notag \\ 
& +\nor{t\,\mathcal{D}^{\theta} \bigl(F(t,\xi) \xi |\xi| \widehat{\phi}(\xi)\bigr)}{} \label{firstsix}
\end{align}
and applying (\ref{unoa}), (\ref{derdos}), (\ref{dertres}) and (\ref{dercuatro}) we have that (\ref{firstsix}) is less than
\begin{align}
C_{\theta}t \Bigl( &\nor{\phi}{1}+ t^{-1/3}\nor{\phi}{1}+\bigl(\nor{\phi}{}+\nor{D_x^{\theta}\phi}{}+\nor{|x|^{\theta}\mathcal{H}\phi}{}\bigr) \notag \\
&+t^{-1/3}\bigl(\nor{\phi}{}+\nor{D_x^{\theta}\phi}{}+\nor{|x|^{\theta}\phi}{}\bigr)+ t^{-2/3}\bigl(\nor{\phi}{}+\nor{D_x^{\theta}\phi}{}+\nor{|x|^{\theta}\phi}{}\bigr)\Bigr) \notag \\
&\leq  C_{\theta} \bigl(\nor{\phi}{1} +\nor{|x|^{\theta}\phi}{} +\nor{|x|^{\theta}\mathcal{H}\phi}{}\bigr). \label{firstseven}
\end{align}
Finally, since $\theta\in (0,1/2)$, $|x|^{\theta}\in A_2$ which means that $\nor{|x|^{\theta}\mathcal{H}\phi}{}\leq c \nor{|x|^{\theta}\phi}{}$, hence
\begin{equation}\label{firsteight}
\nor{t|x|^{\theta}S(t)\bigl(\mathcal{H}+2D_x-3D_x\partial_x \bigr)\phi}{} \leq C\bigl(\nor{\phi}{b}+\nor{|x|^b\phi}{}\bigr).
\end{equation}
(\ref{firstfive}) and (\ref{firsteight}) complete the proof of (\ref{first}). Now, we are going to obtain (\ref{second}) proceeding in the same way. So, applying (\ref{dertres}), (\ref{dercuatro}) and (\ref{simplifica}), we find that
\begin{align}
\nor{\mathcal{D}^{\theta} \bigl(F(t,\xi)\partial_{\xi}(\xi \widehat{\phi}(\xi))\bigr)}{}&\leq \nor{\mathcal{D}^{\theta} \bigl(F(t,\xi)\widehat{\phi}(\xi)\bigr)}{}+\nor{\mathcal{D}^{\theta} \bigl(F(t,\xi)\xi \partial_{\xi}\widehat{\phi}(\xi)\bigr)}{} \notag \\
&\leq c_{\theta}\bigl(\nor{|x|^{\theta}\phi}{}+\nor{\phi}{\theta}\bigr)+c_{\theta}t^{-1/3}\Bigl(\nor{|x|^{\theta}x \phi}{}+\nor{\langle \xi \rangle^{\theta}\partial_{\xi}\widehat{\phi}(\xi)}{}\Bigr) \notag \\
&\leq C_{\theta}t^{-1/3}\bigl(\nor{\phi}{b}+\nor{|x|^{b}\phi}{}\bigr), \label{seconduno}
\end{align}
and since
\begin{align}
\nor{\mathcal{D}^{\theta}\Bigl(\partial_{\xi}(F(t,\xi))\xi \widehat{\phi}\Bigr)}{}&=t\nor{\mathcal{D}^{\theta}\Bigl(F(t,\xi)(sgn(\xi)+2i|\xi|-3\xi |\xi|)\xi \widehat{\phi}\Bigr)}{} \notag \\
&\leq Ct(I_1+I_2+I_3), \label{seconddos}
\end{align}
applying (\ref{dercuatro}), we obtain for $\lambda=1$, $2$, $3$ 
\begin{align}
I_{\lambda}=\nor{\mathcal{D}^{\theta}\Bigl(F(t,\xi)|\xi|^{\lambda}\widehat{\phi}\Bigr)}{}&\leq c_{\theta}t^{-\lambda/3}\bigl(\nor{\phi}{}+\nor{D_x^{\theta}\phi}{}+\nor{|x|^{\theta}\phi}{}\bigr) \notag \\
&\leq c_{\theta}t^{-\lambda/3}\bigl(\nor{\phi}{b}+\nor{|x|^b\phi}{}\bigr). \label{secondtres}
\end{align}
Hence,
\begin{align}
\nor{|x|^{\theta}xS(t)\partial_x\phi}{}&\leq \nor{\mathcal{D}^{\theta} \bigl(F(t,\xi)\partial_{\xi}(\xi \widehat{\phi}(\xi))\bigr)}{}+\nor{\mathcal{D}^{\theta}\Bigl(\partial_{\xi}(F(t,\xi))\xi \widehat{\phi}\Bigr)}{} \notag \\
&\leq C_{\theta}t^{-1/3}\bigl(\nor{\phi}{b}+\nor{|x|^{b}\phi}{}\bigr)+C_{\theta}\bigl(\nor{\phi}{b}+\nor{|x|^{b}\phi}{}\bigr) \notag \\
&\leq C_{\theta}t^{-1/3}\bigl(\nor{\phi}{b}+\nor{|x|^{b}\phi}{}\bigr). \notag
\end{align}
\end{proof}
\begin{prop}\label{xbuenl21<b<3/2}
Let $\theta \in (0,1/2)$, $b=1+\theta$ and $u$ be the solution of the integral equation (\ref{intequation}) with $\phi \in H^b(\mathbb{R})$. If $|x|^b\phi \in L^2(\mathbb{R})$ then $|x|^bu(t)\in L^2(\mathbb{R})$ for all $t\in [0,T]$. 
\end{prop}
\begin{proof}
The proof is the same proof of Proposition \ref{xbuenl21/2<b<1} but applying (\ref{first}) and (\ref{second}) instead of (\ref{dertres}) and (\ref{dercuatro}). 
\end{proof}
\begin{lema}\label{nota2}
Let $\theta \in (1/2,3/2)$. Then,
\begin{equation}\label{noa2}
\nor{|x|^{\theta}\mathcal{H}\phi}{}\leq \nor{|x|^{\theta}\phi}{}\quad \Longleftrightarrow \quad \widehat{\phi}(0)=0.
\end{equation}
If $\theta=1/2$, 
\begin{equation}\label{noa2uno}
\widehat{\phi}(0)=0\quad \Longrightarrow \quad \nor{|x|^{1/2}\mathcal{H}\phi}{}\leq \nor{\langle x \rangle \phi }{}.
\end{equation} 
\end{lema}
\begin{proof}
Let $1/2<\theta <3/2$. Since $x\mathcal{H}\phi=\mathcal{H}(x\phi)$ if and only if $\widehat{\phi}(0)=0$, then
\begin{align}
\nor{|x|^{\theta}\mathcal{H}\phi}{}&=\nor{|x|^{\theta-1}x\mathcal{H}\phi}{}=\nor{|x|^{\theta-1}\mathcal{H}(x\phi)}{} \notag \\
&\leq \nor{|x|^{\theta-1}(x\phi)}{}= \nor{|x|^{\theta}\phi}{}. \label{noa2dos} 
\end{align}
The inequality in (\ref{noa2dos}) is true because $-1/2<\theta -1 <1/2$ and so, $|x|^{\theta -1}\in A_2$. If $\theta =1/2$, with the help of (\ref{noa2}) we obtain that
\begin{align}
\nora{|x|^{1/2}\mathcal{H}\phi}{}{2}&=\int |x|\mathcal{H}\phi \,\overline{\mathcal{H}\phi}\,dx \leq \nor{x\mathcal{H}\phi}{}\nor{\mathcal{H}\phi}{} \notag \\
&=\nor{\mathcal{H}(x\phi)}{}\nor{\phi}{}=\nor{x\phi}{}\nor{\phi}{}\leq \nora{\langle x \rangle \phi }{}{2}. \label{noa2tres}
\end{align} 
\end{proof}
\begin{prop}\label{xbuenl23/2<b<2}
Let $\theta \in [1/2,1)$, $b=1+\theta$ and $u$ be the solution of (\ref{intequation}) with $\phi \in H^b(\mathbb{R})$. If $\widehat{\phi}(0)=0$ and $|x|^b\phi \in L^2(\mathbb{R})$ then $|x|^bu(t)\in L^2(\mathbb{R})$ for all $t\in [0,T]$. 
\end{prop}
\begin{proof}
The argument to prove this proposition is exactly the same that that of Proposition \ref{xbuenl21<b<3/2} except in the estimate to obtain (\ref{firsteight}) because $|x|^{\theta}$ is not an $A_2$ weight. But, applying Lema \ref{nota2} we can obtain (\ref{firsteight}) for $\theta \in [1/2,1)$.
\end{proof}
\begin{lema}
Let $\theta \in (0,1/2)$, $b=2+\theta$ and $\phi \in \dot{Z}_{b,b}$. Then, for any $0< t \leq 1$, it holds that
\begin{align}
\nor{|x|^{2+\theta}S(t)\phi}{}&\leq C_b\,t\,\bigl(\nor{|x|^{2+\theta}\phi}{}+\nor{\phi}{2+\theta}\bigr)\leq C_b \bigl(\nor{|x|^b\phi}{}+\nor{\phi}{b}\bigr) \label{third} \\
\text{and}\quad \quad \nor{|x|^{2+\theta}S(t)\partial_x\phi}{}&\leq C_b\,t^{-1/3}\,\bigl(\nor{|x|^b\phi}{}+\nor{\phi}{b}\bigr). \label{fourth}
\end{align}
\end{lema}
\begin{proof}
We note that
\begin{align}
|x|^bS(t)\phi &= |x|^{\theta}x^2S(t)\phi \notag \\
&= |x|^{\theta}S(t)(x^2\phi)+2|x|^{\theta}\bigl(\partial_{\xi}F(t,\xi)\,\partial_{\xi}\widehat{\phi}(\xi)\bigr)^{\vee}(x)+|x|^{\theta}\bigl(\partial_{\xi}^2(F(t,\xi))\widehat{\phi}(\xi)\bigr)^{\vee}(x) \notag \\
&= B_1+B_2+B_3. \label{thirduno}
\end{align}
Employing $\mathcal{D}^{\theta}$ is sufficient to estimate the $L^2$-norm for the terms $B_1$, $B_2$ and $B_3$. Using (\ref{dertres}) we obtain that for $B_1$
\begin{align}
\nor{\mathcal{D}^{\theta}(F(t,\xi)\partial_{\xi}^2\widehat{\phi})}{}&\leq c_{\theta}\Bigl(\nor{\partial_{\xi}^2\widehat{\phi}}{}+\nor{|\xi|^{\theta}\partial_{\xi}^2\widehat{\phi}}{}+\nor{\mathcal{D}^{\theta}\bigl(\partial_{\xi}^2\widehat{\phi}\,\bigr)}{}\Bigr) \notag \\
&\leq c_{\theta} \Bigl(\nor{x^2\phi}{}+\nor{\langle \xi \rangle^{\theta}J_{\xi}^2 \widehat{\phi}}{}+\nor{|x|^{2+\theta}\phi}{}\Bigr) \notag \\
&\leq c_{\theta} \Bigl(\nor{\langle x\rangle^{2+\theta}\phi}{}+\nora{\langle \xi \rangle^{2+\theta} \widehat{\phi}}{}{\frac{\theta}{2+\theta}}\nora{J_{\xi}^{2+\theta} \widehat{\phi}}{}{\frac{2}{2+\theta}}\Bigr) \notag \\
&\leq C_{\theta} \bigl(\nor{|x|^b\phi}{}+\nor{\phi}{b}\bigr). \label{thirddos}
\end{align}
To estimate the $L^2$-norm of $B_2$ we proceed in a similar way as we estimated the second term of the right hand side of (\ref{firsttwo}) but with $\partial_{\xi}\widehat{\phi}$ instead of $\widehat{\phi}$. So, from (\ref{firstseven}), applying (\ref{simplifica}) and since $|x|^{\theta}\in A_2$, we obtain that 
\begin{align}
\nor{\mathcal{D}^{\theta}\Bigl(\partial_{\xi}F(t,\xi)\,\partial_{\xi}\widehat{\phi}(\xi)\Bigr)}{}&\leq c_{\theta} \Bigl(\nor{x\phi}{} +\nor{|x|^{\theta}x\phi}{} +\nor{|x|^{\theta}\mathcal{H}(x\phi)}{}+\nor{|\xi|^{\theta}\partial_{\xi}\widehat{\phi}}{}\Bigr) \notag \\
&\leq C_{\theta}\bigl(\nor{|x|^b\phi}{}+\nor{\phi}{b}\bigr). \label{thirdtres}
\end{align}
To estimate the $L^2$-norm of $B_3$ we use (\ref{dos}) and that the product $\delta \widehat{\phi}=\widehat{\phi}(0)=0$. So,
\begin{align}
&\nor{\mathcal{D}^{\theta}\bigl(2 t\delta +tF(t,\xi)[2i\,sgn(\xi)-6|\xi|]+t^2F(t,\xi)[sgn(\xi)+2i|\xi|-3 \xi|\xi|]^2\bigr)\widehat{\phi}}{} \notag \\
&\leq t\nor{\mathcal{D}^{\theta}\bigl(F(t,\xi)[2i\,sgn(\xi)-6|\xi|]\widehat{\phi}\bigr)}{}+t^2\nor{\mathcal{D}^{\theta}\bigl(F(t,\xi)[sgn(\xi)+2i|\xi|-3 \xi|\xi|]^2\widehat{\phi}\bigr)}{} \notag \\
&=tB_{31}+t^2B_{32}. \label{thirdcuatro}
\end{align}
Then, applying (\ref{dertres}), (\ref{dercuatro}) and the fact that $|x|^{\theta}\in A_2$, 
\begin{align}
tB_{31}&\leq ct\Bigl(\nor{\mathcal{D}^{\theta}\bigl(F(t,\xi) sgn(\xi) \widehat{\phi}\,\bigr)}{}+\nor{\mathcal{D}^{\theta}\bigl(F(t,\xi) |\xi| \widehat{\phi}\,\bigr)}{}\Bigr) \notag \\
&\leq c_{\theta} t\bigl(\nor{\phi}{}+\nor{D_x^{\theta}\phi}{}+\nor{|x|^{\theta}\phi}{}\bigr)+c_{\theta} t^{2/3}\bigl(\nor{\phi}{}+\nor{D_x^{\theta}\phi}{}+\nor{|x|^{\theta}\phi}{}\bigr) \notag \\
&\leq C_{\theta}\bigl(\nor{|x|^b\phi}{}+\nor{\phi}{b}\bigr), \label{thirdcinco}
\end{align}
and
\begin{align}
t^2B_{32}&= t^2\nor{\mathcal{D}^{\theta}\bigl(F(t,\xi)(sgn(\xi)-3 \xi|\xi|)^2\widehat{\phi}-4F(t,\xi)\xi^2\widehat{\phi}+4iF(t,\xi)|\xi|(sgn(\xi)-3 \xi|\xi|)\widehat{\phi}\,\bigr)}{} \notag \\
&\leq c_{\theta}t^2\sum_{j=0}^4\nor{\mathcal{D}^{\theta}\bigl(F(t,\xi)\xi^j\widehat{\phi}\,\bigr)}{}\leq c_{\theta}t^2\sum_{j=0}^4t^{-j/3}\bigl(\nor{\phi}{}+\nor{D_x^{\theta}\phi}{}+\nor{|x|^{\theta}\phi}{}\bigr) \notag \\
&\leq C_{\theta}\bigl(\nor{|x|^b\phi}{}+\nor{\phi}{b}\bigr). \label{thirdseis}
\end{align}
Hence, (\ref{thirddos}), (\ref{thirdtres}), (\ref{thirdcinco}) and (\ref{thirdseis}) imply (\ref{third}). To prove (\ref{fourth}) we note that
\begin{align}
|x|^b&S(t)\partial_x\phi = |x|^{\theta}x^2S(t)\partial_x\phi \notag \\
&= |x|^{\theta}S(t)(x^2\partial_x\phi)+2|x|^{\theta}\bigl(\partial_{\xi}F(t,\xi)\,\partial_{\xi}\widehat{\partial_x\phi}(\xi)\bigr)^{\vee}(x)+|x|^{\theta}\bigl(\partial_{\xi}^2(F(t,\xi))\widehat{\partial_x\phi}(\xi)\bigr)^{\vee}(x) \notag \\
&= G_1+G_2+G_3. \label{fourthuno}
\end{align}
We proceed exactly in the same way in which we did the proof of (\ref{third}). Then, to estimate the $L^2$-norm of $G_2+G_3$ we use the inequalities applied to obtain (\ref{thirdtres}), (\ref{thirdcinco}) and (\ref{thirdseis}) but substituting $\phi$ by $\partial_x\phi$. So, applying (\ref{simplificaotra}) we have
\begin{align}
&\nor{\mathcal{D}^{\theta}\Bigl(\partial_{\xi}F(t,\xi)\,\partial_{\xi}\bigl(\xi \widehat{\phi}(\xi)\bigr)\Bigr)}{}+\nor{\mathcal{D}^{\theta}\bigl(\partial_{\xi}^2F(t,\xi)\bigr) \xi \widehat{\phi}}{} \notag \\
&\leq c_{\theta}\bigl(\nor{|x|^{1+\theta}\partial_x\phi}{}+\nor{\partial_x\phi}{1+\theta}\bigr) \notag \\
&\leq C_{\theta}\bigl(\nor{|x|^b\phi}{}+\nor{\phi}{b}\bigr). \label{fourthdos}
\end{align}
Finally, to estimate the $L^2$-norm of $G_1$ we find that
\begin{align}
\nor{\mathcal{D}^{\theta}(F(t,\xi)\partial_{\xi}^2(\xi \widehat{\phi}))}{}&\leq 2\nor{\mathcal{D}^{\theta}(F(t,\xi)\partial_{\xi}\widehat{\phi})}{}+\nor{\mathcal{D}^{\theta}(F(t,\xi)\xi\partial_{\xi}^2\widehat{\phi})}{} \notag \\
&\leq G_{11} + G_{12}, \label{fourthtres}
\end{align}
applying (\ref{dertres}) but substituting $h$ by $x\phi$ we have that
\begin{align}
G_{11}&\leq c_{\theta} \Bigl(\nor{x\phi}{}+\nor{|\xi|^{\theta}\partial_{\xi}\widehat{\phi}}{}+\nor{\mathcal{D}^{\theta}\partial_{\xi}\widehat{\phi}}{}\Bigr) \notag \\
&\leq C_{\theta}\bigl(\nor{|x|^b\phi}{}+\nor{\phi}{b}\bigr). \label{fourthcuatro}
\end{align}
and applying (\ref{dercuatro}) but substituting $h$ by $x^2\phi$ we get that
\begin{align}
G_{12}&=\nor{\mathcal{D}^{\theta}\bigl(F(t,\xi)\xi \,\widehat{x^2\phi}\,\bigr)}{} \notag \\
&\leq c_{\theta} t^{-1/3}\bigl(\nor{|x|^{\theta}x^2\phi}{}+\nor{J_x^{\theta}x^2\phi}{}\bigr) \notag \\
&\leq C_{\theta}t^{-1/3}\bigl(\nor{|x|^b\phi}{}+\nor{\phi}{b}\bigr), \label{fourthcinco}
\end{align}
where we have used the same inequalities applied to obtain (\ref{thirddos}) because $\nor{J_x^{\theta}x^2\phi}{}=\nor{\langle \xi \rangle^{\theta}\partial_{\xi}^2\widehat{\phi}}{}$.
\end{proof}
\begin{prop}\label{xbuenl22<b<5/2}
Let $\theta \in (0,1/2)$, $b=2+\theta$ and $u$ be the solution of the integral equation (\ref{intequation}) with $\phi \in H^b(\mathbb{R})$. If $\widehat{\phi}(0)=0$ and $|x|^b\phi \in L^2(\mathbb{R})$ then $|x|^bu(t)\in L^2(\mathbb{R})$ for all $t\in [0,T]$.
\end{prop}
\begin{proof}
The proof is the same proof of Proposition \ref{xbuenl21/2<b<1} but applying (\ref{third}) and (\ref{fourth}) instead of (\ref{dertres}) and (\ref{dercuatro}). 
\end{proof}
\begin{proof}[Proof of Theorem \ref{pre}]
Part \textbf{(i)} is direct consequence of Propositions \ref{xbuenl2}, \ref{xbuenl21/2<b<1} and \ref{xbuenl21<b<3/2}. Part \textbf{(ii)} is deduced from Propositions \ref{xbuenl23/2<b<2} and \ref{xbuenl22<b<5/2}.
\end{proof}

\setcounter{equation}{0}
\section{Proof of Theorem \ref{contunica1}} 

Without loss of generality we assume that $t_1=0<t_2$. Since $u(t_1)=\phi \in Z_{3/2,3/2}$, $\phi \in Z_{3/2,b}$ where $b<3/2$, and then $u\in C([0,T]; Z_{3/2,3/2-})$ by Proposition (\ref{xbuenl21<b<3/2}). The solution to the IVP npBO (\ref{npbo}) can be represented by Duhamel's formula
\begin{equation}\label{uDuhformula}
u(t)=S(t)\phi - \int_0^tS(t-t')(uu_x)(t')\,dt',
\end{equation}
where $S(t)$ is given by (\ref{semigroup}). From Plancherel's equality we have that for every $t$, $|x|^{1/2}x S(t)\phi \in L^2(\mathbb{R})$ if and only if $D_{\xi}^{1/2}\partial_{\xi}(F_{\mu}(t,\xi)\widehat{\phi}(\xi))\in L^2(\mathbb{R})$. The argument in our proof requires localizing near the origin in Fourier frequencies by a function $\chi \in C_0^{\infty}$, $\text{supp} \,\chi \subseteq (-\epsilon , \epsilon)$ and $\chi \equiv 1$ on $(-\epsilon/2, \epsilon/2)$.
Let us start with the computation for the linear part in (\ref{uDuhformula}) by introducing a commutator as follows
\begin{align}
\chi D_{\xi}^{1/2}\partial_{\xi}(F_{\mu}(t,\xi)\widehat{\phi}(\xi))&=\left[\chi , D_{\xi}^{1/2}\right] \partial_{\xi}\left(F_{\mu}(t,\xi)\widehat{\phi}(\xi)\right)+D_{\xi}^{1/2}\left(\chi \partial_{\xi}(F_{\mu}(t,\xi)\widehat{\phi}(\xi))\right) \notag \\
&=A+B. \label{linearpart}
\end{align} 
From Proposition \ref{simplestimate} and identity (\ref{uno}) we have that
\begin{align}
\nor{A}{}&= \nor{[\chi , D_{\xi}^{1/2}] \partial_{\xi}(F_{\mu}(t,\xi)\widehat{\phi}(\xi))}{}  \notag \\
&\lesssim \nor{\partial_{\xi}(F_{\mu}(t,\xi)\widehat{\phi}(\xi))}{}  \notag  \\
&\lesssim \nor{\mu t \,\text{sgn}(\xi)F_{\mu}(t,\xi)\widehat{\phi}(\xi)}{}+\nor{2it|\xi|F_{\mu}(t,\xi)\widehat{\phi}(\xi)}{}+\nor{3\mu t \xi |\xi| F_{\mu}(t,\xi)\widehat{\phi}(\xi)}{}+\nor{F_{\mu}(t,\xi)\partial_{\xi} \widehat{\phi}(\xi)}{}  \notag \\
&\lesssim te^{\mu t}\nor{\phi}{}+2t(e^{\mu t}+(\mu t)^{-1/3})\nor{\phi}{}+3t (e^{\mu t}+(\mu t)^{-2/3})\nor{\phi}{}+e^{\mu t}\nor{\partial_{\xi}\widehat{\phi}(\xi)}{} \notag \\
&\lesssim [(1+t)e^{\mu t}+t^{2/3}+t^{1/3}]\,\nor{\phi}{Z_{1,1}}, \label{aestimate}
\end{align}
where were used (\ref{cotaemu}) and (\ref{unoa}). Rewriting $B$, we obtain that
\begin{align}
B&=D_{\xi}^{1/2}(\chi \partial_{\xi}(F_{\mu}(t,\xi)\widehat{\phi}(\xi))) \notag \\
&=D_{\xi}^{1/2}\left(\mu t \,\text{sgn}(\xi) \chi F_{\mu}(t,\xi)\widehat{\phi}(\xi)\right)+D_{\xi}^{1/2}\left(2i t |\xi| \chi F_{\mu}(t,\xi)\widehat{\phi}(\xi)\right)+\notag \\
&+D_{\xi}^{1/2}\left((-3\mu) t \xi |\xi| \chi F_{\mu}(t,\xi)\widehat{\phi}(\xi)\right)+D_{\xi}^{1/2}\left(\chi F_{\mu}(t,\xi) \partial_{\xi} \widehat{\phi}(\xi)\right) \notag \\
&=B_1+B_2+B_3+B_4. \label{bestimate}
\end{align} 
Now, we are going to estimate $B_4$ in $L^2(\mathbb{R})$. From Theorem \ref{derivaStein}, inequalities (\ref{cotaemu}), (\ref{productostein}), in the Lemma \ref{rprod}, and the inequality (\ref{dertres}), in the Lemma \ref{clave1}, we get that 
\begin{align}
\nor{B_4}{}&\lesssim \nor{\chi F_{\mu}(t,\xi) \partial_{\xi} \widehat{\phi}(\xi)}{}+\nor{\mathcal{D}_{\xi}^{1/2}\left( F_{\mu}(t,\xi) \chi \partial_{\xi} \widehat{\phi}(\xi)\right)}{} \notag \\
&\lesssim e^{\mu t}\nor{x \phi}{}+\nor{\chi \partial_{\xi} \widehat{\phi}(\xi)}{}+\nor{|\xi|^{1/2}\chi \partial_{\xi} \widehat{\phi}(\xi)}{}+\nor{\mathcal{D}_{\xi}^{1/2}\left( \chi \partial_{\xi} \widehat{\phi}(\xi)\right)}{} \notag \\
&\lesssim e^{\mu t}\nor{x \phi}{}+\nor{\chi}{\infty}\nor{x \phi}{}+\nor{|\xi|^{1/2}\chi}{\infty}\nor{x \phi}{}+\nor{\mathcal{D}_{\xi}^{1/2}\left(\chi \right)\,\partial_{\xi} \widehat{\phi}(\xi)}{}+\nor{\chi  \mathcal{D}_{\xi}^{1/2}\left( \partial_{\xi} \widehat{\phi}(\xi)\right)}{} \notag \\
&\leq c(T) \nor{\langle x \rangle^{1+1/2} \phi}{}. \label{b4}
\end{align}
Estimates for $B_2$ and $B_3$ in $L^2(\mathbb{R})$ are obtained in a similar way but using (\ref{dercuatro}) instead of (\ref{dertres}). To estimate $B_1$ in $L^2(\mathbb{R})$ we introduce $\tilde{\chi} \in C^{\infty}_0(\mathbb{R})$ such that $\tilde{\chi}\equiv 1$ on $\text{supp}\,(\chi)$. Then, we can express this term as
\begin{align}
D_{\xi}^{1/2}\left( t \,\text{sgn}(\xi) F_{\mu}(t,\xi) \, \chi \, \widehat{\phi}(\xi)\right)&= t D_{\xi}^{1/2}\left( F_{\mu}(t,\xi)\, \tilde{\chi}\, \chi \, \text{sgn}(\xi)\, \widehat{\phi}(\xi)\right) \notag \\
&= t \left(\left[D_{\xi}^{1/2}, F_{\mu}(t,\xi) \,\tilde{\chi} \right] \, \chi \,\text{sgn}(\xi) \,\widehat{\phi}(\xi) + F_{\mu}(t,\xi)\, \tilde{\chi} \, D_{\xi}^{1/2}\bigl( \chi \, \text{sgn}(\xi) \widehat{\phi}(\xi)\bigr)\right) \notag \\
&=t(B_{1,1}+B_{1,2}). \label{b1}
\end{align}
Again, Proposition \ref{simplestimate} can be applied to estimate $B_{1,1}$ in $L^2(\mathbb{R})$ as
\begin{align}
\nor{B_{1,1}}{}&=\nor{\left[D_{\xi}^{1/2}, F_{\mu}(t,\xi) \,\tilde{\chi} \right] \, \chi \, \text{sgn}(\xi) \,\widehat{\phi}(\xi)}{} \notag \\
&\lesssim \nor{\chi \, \text{sgn}(\xi) \,\widehat{\phi}(\xi)}{} \notag \\
&\lesssim \nor{\phi}{}. \label{b11}
\end{align}
Once we show that the integral part in Duhamel's formula (\ref{uDuhformula}) lies in $L^2(|x|^3\,dx)$, we will be able to conclude that
$$B_{1,2}, \; \tilde{\chi} \, D_{\xi}^{1/2}\bigl( \chi \, \text{sgn}(\xi) \widehat{\phi}(\xi)\bigr), \; D_{\xi}^{1/2}\bigl( \tilde{\chi} \, \chi \, \text{sgn}(\xi) \widehat{\phi}(\xi)\bigr) \; \in L^2(\mathbb{R}), $$
because $u(t_2)=u(t)\in Z_{3/2, 3/2}$ by hypothesis. Therefore, from Proposition \ref{I1} it will follow that $\widehat{\phi}(0)=0$, and from the conservation law 
$$I(u)=\int_{\mathbb{R}}u(x,t)\,dx=\widehat{\phi}(0)=0$$
i. e., $\widehat{u}(0,t)=0$ for all $t$. Hence, $u(\cdot , t) \in \dot{Z}_{3/2,3/2}$. \\ \\
In order to complete the proof, we consider the integral part in Duhamel's formula. We will denote $z=uu_x=\frac{1}{2}\,\partial_x(u^2)$ and so $\widehat{z}=i\,\frac{\xi}{2}\,\widehat{u}*\widehat{u}.$
\begin{align}
\left(|x|^{1/2}\,x\,\int_0^tS(t-t')z(t')\,dt'\right)^{\wedge} (\xi)&=\int_0^tD_{\xi}^{1/2}\partial_{\xi}\bigl( F_{\mu}(t-t',\xi)\widehat{z}(t',\xi)\bigr)\,dt' \notag \\
&=\int_0^tD_{\xi}^{1/2}\bigl(\partial_{\xi} F_{\mu}(t-t',\xi) \, \widehat{z}(t',\xi)\bigr)\,dt' +\int_0^tD_{\xi}^{1/2}\bigl( F_{\mu}(t-t',\xi) \, \partial_{\xi}\widehat{z}(t',\xi) \bigr)\,dt' \notag \\
&=\mathcal{A}+\mathcal{B}. \label{integralpart}
\end{align}
We localize again with the help of $\chi \in C_0^{\infty}(\mathbb{R})$ and then we can write
\begin{align}
\chi \,\mathcal{A}&=\int_0^t\left[\chi , D_{\xi}^{1/2} \right] \bigl(\partial_{\xi} F_{\mu}(t-t',\xi) \, \widehat{z}(t',\xi)\bigr)\,dt' + \int_0^tD_{\xi}^{1/2}\bigl( \chi \, \partial_{\xi} F_{\mu}(t-t',\xi) \, \widehat{z}(t',\xi)\bigr)\,dt' \notag \\
&=\int_0^t\left[\chi , D_{\xi}^{1/2} \right] \bigl((t-t')(\mu \text{sgn}(\xi)+2i|\xi|-3\mu \xi |\xi|) F_{\mu}(t-t',\xi)\, \widehat{z}(t',\xi)\bigr)\,dt' + \notag \\
&\quad + \int_0^tD_{\xi}^{1/2}\bigl( \chi (t-t')(\mu \text{sgn}(\xi)+2i|\xi|-3\mu \xi |\xi|) F_{\mu}(t-t',\xi)\, \widehat{z}(t',\xi)\bigr)\,dt' \notag \\
&=\mathcal{A}_1+\mathcal{A}_2+\mathcal{A}_3+\mathcal{A}_4+\mathcal{A}_5+\mathcal{A}_6. \label{chiA}
\end{align}
and
\begin{align}
\chi \,\mathcal{B}&=\int_0^t \chi D_{\xi}^{1/2}\bigl( F_{\mu}(t-t',\xi) \, \partial_{\xi}\widehat{z}(t',\xi) \bigr)\,dt' \notag \\
&=\int_0^t[\chi , D_{\xi}^{1/2}]\bigl( F_{\mu}(t-t',\xi) \, \partial_{\xi}\widehat{z}(t',\xi) \bigr)\,dt' + \int_0^tD_{\xi}^{1/2}\bigl( \chi F_{\mu}(t-t',\xi) \, \partial_{\xi}\widehat{z}(t',\xi) \bigr)\,dt' \notag \\
&=\mathcal{B}_1+\mathcal{B}_2. \label{chiB}
\end{align}
Now, we must bound all terms in (\ref{chiA}) and (\ref{chiB}). But, we limit our attention to the terms $\mathcal{A}_3$, $\mathcal{A}_6$, $\mathcal{B}_1$ and $\mathcal{B}_2$ which are more representatives and the others can be treated in a similar way. So, combining Proposition \ref{simplestimate}, (\ref{unoa}) and Holder's inequality we have that
\begin{align}
\nor{\mathcal{A}_3}{}&\leq  \int_0^t\nor{ \left[\chi , D_{\xi}^{1/2} \right] \bigl(-3\mu (t-t') \xi |\xi| \, F_{\mu}(t-t',\xi)\, \widehat{z}(t',\xi)\bigr)}{}\,dt' \notag \\
&\lesssim \int_0^t(t-t') \nor{ F_{\mu}(t-t',\xi)\,\xi |\xi|\, \widehat{z}(t',\xi)}{}\,dt' \notag \\
&\lesssim \int_0^t(t-t')\bigl(e^{\mu (t-t')}+(t-t')^{-2/3}\bigr) \, \nor{\xi \, \widehat{u}*\widehat{u}(t',\xi)}{}\,dt' \notag \\
&\lesssim \Bigl(\int_0^t\bigl((t-t')e^{\mu (t-t')}+(t-t')^{1/3}\bigr)^2 \,dt'\Bigr)^{1/2} \nor{\partial_x (u^2)}{L_T^2L_x^2} \notag \\
&\lesssim c(T) T^{1/2}\nor{u}{L_T^{\infty}L_x^{\infty}}\nor{\partial_xu}{L_T^{\infty}L_x^2} \notag \\
&\lesssim c(T) \nora{u}{L_T^{\infty}H_x^1}{2}. \label{cotaA3} 
\end{align} 
For $\mathcal{A}_6$, using Stein's derivative, (\ref{unoa}) and (\ref{dertres}), we obtain that
\begin{align}
\nor{\mathcal{A}_6}{}&\leq  \int_0^t\nor{ D_{\xi}^{1/2} \Bigl(-3\mu (t-t') \chi \, F_{\mu}(t-t',\xi)\,\xi |\xi| \, \widehat{z}(t',\xi)\Bigr)}{}\,dt' \notag \\
&\lesssim \int_0^t (t-t') \nor{ \chi \, \xi |\xi| \, F_{\mu}(t-t',\xi)\, \widehat{z}(t',\xi)}{}\,dt' + \int_0^t (t-t') \nor{ \mathcal{D}_{\xi}^{1/2} \Bigl(F_{\mu}(t-t',\xi)\,\xi |\xi| \,\chi \, \widehat{z}(t',\xi)\Bigr)}{}\,dt' \notag \\
&\lesssim \int_0^t (t-t') \nor{ \chi \,\xi |\xi|}{\infty} e^{\mu(t-t')}\nor{\widehat{z}}{}\,dt' + \int_0^t (t-t') \Bigl( \nor{\chi \,\xi |\xi|\,\widehat{z}}{} +\nor{|\xi|^{1/2}\,\chi \,\xi |\xi|  \, \widehat{z}}{} +\nor{ \mathcal{D}_{\xi}^{1/2} \bigl(\chi \,\xi |\xi|\,\widehat{z}\, \bigr)}{}\Bigr)\,dt'  \notag \\
&=\mathcal{Y}_1+\mathcal{Y}_2.
\end{align}
Almost repeating the estimates to obtain (\ref{cotaA3}) one has that 
$$\mathcal{Y}_1\leq c(T) \nora{u}{L_T^{\infty}H_x^1}{2}\, ,$$
and using (\ref{productostein}) and (3.12) from \cite{LP}
\begin{align}
\mathcal{Y}_2 &\leq \int_0^t (t-t') \Bigr(  \Bigl( \nor{\chi \,\xi^2 |\xi|}{\infty} +\nor{\chi \,\xi^2 |\xi|^{3/2}}{\infty}+\nor{ \mathcal{D}_{\xi}^{1/2} \bigl(\chi \,\xi^2 |\xi|\bigr)}{\infty}\Bigr)\nor{\widehat{u}*\widehat{u}}{} +\nor{ \chi \,\xi^2 |\xi|}{\infty}\nor{\mathcal{D}_{\xi}^{1/2} \bigl(\widehat{u}*\widehat{u}\bigr)}{}\Bigr)\,dt'  \notag \\
&\lesssim c(T)\Bigl(\nor{\widehat{u}*\widehat{u}}{L_T^1L_x^2}+\nor{\mathcal{D}_{\xi}^{1/2} \bigl(\widehat{u}*\widehat{u}\bigr)}{L_T^1L_x^2}\Bigr) \notag \\
&\lesssim c(T)\Bigl(\nor{u^2}{L_T^1L_x^2}+\nor{|x|^{1/2} u^2}{L_T^1L_x^2}\Bigr) \notag \\
&\lesssim c(T)\Bigl(T \nor{u}{L_T^{\infty}L_x^{\infty}}\nor{u}{L_T^{\infty}L_x^2}+T \nor{u}{L_T^{\infty}L_x^{\infty}}\nor{|x|^{1/2} u}{L_T^{\infty}L_x^2}\,\Bigr) \notag \\
&\lesssim c(T) \nor{u}{L_T^{\infty}H_x^1} \Bigl(\nor{u}{L_T^{\infty}H_x^1}+\nor{|x|^{1/2}u}{L_T^{\infty}L_x^2}\,\Bigr). \label{y2}
\end{align}
For $\mathcal{B}_1$, applying Proposition \ref{simplestimate}, (\ref{unoa}) and, again, (3.12) from \cite{LP}, we have
\begin{align}
\nor{\mathcal{B}_1}{}&\lesssim \int_0^t\nor{[\chi , D_{\xi}^{1/2}]\bigl( F_{\mu}(t-t',\xi) \, \partial_{\xi}(\xi \, \widehat{u}*\widehat{u} ) \bigr)}{}\,dt' \notag \\
&\lesssim \int_0^t\nor{ F_{\mu}(t-t',\xi) \, \widehat{u}*\widehat{u} }{}\,dt' + \int_0^t\nor{ F_{\mu}(t-t',\xi) \,\xi \, \partial_{\xi}(\widehat{u}*\widehat{u} ) }{}\,dt' \notag \\
&\lesssim \int_0^t e^{\mu (t-t')} \nor{\widehat{u}*\widehat{u} }{}\,dt' + \int_0^t \Bigl(e^{\mu (t-t')} + (t-t')^{-1/3}\Bigr) \nor{\partial_{\xi}(\widehat{u}*\widehat{u} ) }{}\,dt' \notag \\
&\lesssim c(T)\Bigl(\nor{u^2}{L_T^1L_x^2}+\nor{x u^2}{L_T^1L_x^2}\Bigr) \notag \\
&\lesssim c(T)\Bigl(\nora{u}{L_T^{\infty}H_x^1}{2} + \nor{x u}{L_T^{\infty}L_x^2}\nor{u}{L_T^{\infty}H_x^1}\,\Bigr) \notag \\
&\lesssim c(T) \nor{u}{L_T^{\infty}H_x^1} \Bigl(\nor{u}{L_T^{\infty}H_x^1}+\nor{x\,u}{L_T^{\infty}L_x^2}\,\Bigr). \label{B1}
\end{align}
Finally, for $\mathcal{B}_2$, we use Stein's derivative
\begin{align}
\nor{\mathcal{B}_2}{} &\lesssim \int_0^t \nor{ \chi \, F_{\mu}(t-t',\xi) \, \partial_{\xi}\widehat{z}(t',\xi)}{}\,dt' + \int_0^t\nor{\mathcal{D}_{\xi}^{1/2}\bigl( \chi \, F_{\mu}(t-t',\xi) \, \partial_{\xi}(\xi \,\widehat{u}*\widehat{u}(t',\xi)) \bigr)}{}\,dt' \notag \\
&=Z_1+Z_2.
\end{align}
Estimate for $Z_1$ is obtained in similar way as it was bounded $\mathcal{B}_1$. To estimate $Z_2$ we use (\ref{dertres}), (\ref{productostein}) and (3.12) from \cite{LP}
\begin{align}
Z_2&\leq \int_0^t\nor{\mathcal{D}_{\xi}^{1/2}\bigl( F_{\mu}(t-t',\xi) \, \chi \, \widehat{u}*\widehat{u} \bigr)}{}\,dt' + \int_0^t\nor{\mathcal{D}_{\xi}^{1/2}\bigl(F_{\mu}(t-t',\xi) \, \chi \, \xi \,\partial_{\xi}(\widehat{u}*\widehat{u}) \bigr)}{}\,dt' \notag \\ 
&\lesssim \int_0^t\Bigl(\nor{\chi \, \widehat{u}*\widehat{u}}{}+\nor{|\xi|^{1/2} \, \chi \, \widehat{u}*\widehat{u}}{}+\nor{\mathcal{D}_{\xi}^{1/2}\bigl( \chi \, \widehat{u}*\widehat{u} \bigr)}{}\Bigr)\,dt' + \notag \\
&\qquad +\int_0^t\Bigl(\nor{\chi \, \xi \,\partial_{\xi}(\widehat{u}*\widehat{u})}{}+\nor{|\xi|^{1/2}\,\chi \, \xi \,\partial_{\xi}(\widehat{u}*\widehat{u})}{}+\nor{\mathcal{D}_{\xi}^{1/2}\bigl(\chi \, \xi \,\partial_{\xi}(\widehat{u}*\widehat{u}) \bigr)}{}\Bigr)\,dt' \notag \\ 
&\lesssim \int_0^t\Bigl(\Bigl(\nor{\chi}{\infty}+\nor{|\xi|^{1/2}\chi}{\infty}+\nor{\mathcal{D}_{\xi}^{1/2}\bigl(\chi \bigr)}{\infty}\Bigr)\nor{\widehat{u}*\widehat{u}}{}+\nor{\chi}{\infty}\nor{\mathcal{D}_{\xi}^{1/2}\bigl(\widehat{u}*\widehat{u} \bigr)}{}\Bigr)\,dt' + \notag \\
&\qquad \int_0^t\Bigl(\Bigl(\nor{\chi \xi}{\infty}+\nor{|\xi|^{1/2}\chi \xi}{\infty}+\nor{\mathcal{D}_{\xi}^{1/2}\bigl(\chi \xi \bigr)}{\infty}\Bigr)\nor{\partial_{\xi}(\widehat{u}*\widehat{u})}{}+\nor{\chi \xi}{\infty}\nor{\mathcal{D}_{\xi}^{1/2}\partial_{\xi}\bigl(\widehat{u}*\widehat{u} \bigr)}{}\Bigr)\,dt' + \notag \\
&\lesssim c(T)\Bigl( \nor{\widehat{u}*\widehat{u}}{L_T^1L_{\xi}^2}+\nor{\mathcal{D}_{\xi}^{1/2}(\widehat{u}*\widehat{u})}{L_T^1L_{\xi}^2}+\nor{\partial_{\xi}(\widehat{u}*\widehat{u})}{L_T^1L_{\xi}^2}+\nor{\mathcal{D}_{\xi}^{1/2}\partial_{\xi}\bigl(\widehat{u}*\widehat{u} \bigr)}{L_T^1L_{\xi}^2}\Bigr) \notag \\
&\lesssim c(T) \nor{u}{L_T^{\infty}H_x^1} \Bigl(\nor{u}{L_T^{\infty}H_x^1}+\nor{|x|^{1/2}u}{L_T^{\infty}L_x^2}+\nor{x u}{L_T^{\infty}L_x^2}+ \nor{|x|^{3/2}u}{L_T^{\infty}L_x^2}\Bigr). \label{z2}
\end{align}
Hence, the terms in (\ref{chiA}) and (\ref{chiB}) are all bounded, then by applying the argument after inequality (\ref{b11}) we complete the proof.


\setcounter{equation}{0}
\section{Proof of Theorem \ref{contunica2}} 

From the Proposition \ref{xbuenl22<b<5/2} and the hypothesis we have that for any $\epsilon >0$
$$u\in C([0,T]; \dot{Z}_{5/2, 5/2-\epsilon})\qquad \text{and} \qquad u(\cdot, t_j)\in L^2(|x|^5\,dx), \qquad j=1, \, 2, \, 3.$$
Consquently,
$$\widehat{u}\in C([0,T]; H^{5/2-\epsilon}(\mathbb{R})\cap L^2(|\xi|^5\,d\xi)) \qquad \text{and}\qquad \widehat{u}(\cdot, t_j)\in H^{5/2}(\mathbb{R}), \qquad j=1, \, 2, \, 3,$$
for all $\epsilon >0$. Thus, in particular it follows that
$$\widehat{u}*\widehat{u} \in C([0,T];H^{4}(\mathbb{R})\cap L^2(|\xi|^5\,d\xi)).$$
Let us assume that $t_1=0<t_2<t_3$. Applying (\ref{uno}) and (\ref{dos}) from Lemma \ref{lemdecaida} we obtain that
\begin{align}
\partial_{\xi}^2\bigl(F_{\mu}(t,\xi)\widehat{\phi}(\xi)\bigr)&=E(t,\xi,\widehat{\phi}(\xi)) \notag \\
&=\bigl[2it\,\text{sgn}(\xi)-6\mu t |\xi| +t^2\mu^2+4i\mu t^2\xi -(6\mu^2+4)t^2 \xi^2-12i\mu t^2 \xi^3 +9\mu^2 t^2 \xi^4\bigr]\,F_{\mu}(t,\xi)\,\widehat{\phi}(\xi) \notag \\
& +2\mu t \,\text{sgn}(\xi)F_{\mu}(t,\xi)\partial_{\xi}\widehat{\phi}(\xi) + 4it |\xi|\,F_{\mu}(t,\xi)\,\partial_{\xi}\widehat{\phi}(\xi)-6\mu t \xi |\xi| \, F_{\mu}(t,\xi)\,\partial_{\xi}\widehat{\phi}(\xi)+F_{\mu}(t,\xi)\,\partial_{\xi}^2\widehat{\phi}(\xi), \label{dosprima}
\end{align}
where we apply that the initial data $\phi$ have zero mean value and for this the term involving the Dirac function in (\ref{dosprima}) vanishes. Using Plancherel's theorem and Duhamel's formula (\ref{uDuhformula}), it will be sufficient to show that the assumption that
\begin{equation}
D_{\xi}^{1/2}E(t,\xi,\widehat{\phi}(\xi))-\int_0^tD_{\xi}^{1/2}E(t-t',\xi,\widehat{z}(t',\xi))\,dt' \, , \label{d5/2fd}
\end{equation}
lies in $L^2(\mathbb{R})$ for times $t_1=0<t_2<t_3$, where $\widehat{z}=i\,\frac{\xi}{2}\,\widehat{u}*\widehat{u}$, leads to a contradiction. First, we prove that the linear part in (\ref{d5/2fd}) persists in $L^2$. We introduce as in the proof of Theorem \ref{contunica1} a localizer $\chi \in C_0^{\infty}$, $\text{supp} \,\chi \subseteq (-\epsilon , \epsilon)$ and $\chi \equiv 1$ on $(-\epsilon/2, \epsilon/2)$ so that
\begin{align}
\chi D_{\xi}^{1/2}\partial_{\xi}^2\bigl(F_{\mu}(t,\xi)\widehat{\phi}(\xi)\bigr)&= [\chi , D_{\xi}^{1/2}]\partial_{\xi}^2\bigl(F_{\mu}(t,\xi)\widehat{\phi}(\xi)\bigr)+ D_{\xi}^{1/2}\bigl(\chi \partial_{\xi}^2\bigl(F_{\mu}(t,\xi)\widehat{\phi}(\xi)\bigr)\bigr) \notag \\
&=J+K. \label{jmask}
\end{align}
As for the first term $J$, from Proposition \ref{simplestimate}, this is bounded in $L^2(\mathbb{R})$ by $\nor{\partial_{\xi}^2\bigl(F_{\mu}(t,\xi)\widehat{\phi}(\xi)\bigr)}{}$, which is finite as can be observed from its explicit representation in (\ref{dosprima}), the assumption on the initial data $\phi$, and the quite similar computation already performed in (\ref{aestimate}), therefore we omit details.

On the other hand, for $J$, we notice that
\begin{align}
K&=D_{\xi}^{1/2}\bigl(\chi \partial_{\xi}^2\bigl(F_{\mu}(t,\xi)\widehat{\phi}(\xi)\bigr)\bigr) \notag \\
&=2itD_{\xi}^{1/2}\bigl(\chi \text{sgn}(\xi)F_{\mu}(t,\xi)\widehat{\phi}(\xi)\bigr)-6\mu tD_{\xi}^{1/2}\bigl(\chi |\xi| F_{\mu}(t,\xi)\widehat{\phi}(\xi)\bigr) +t^2\mu^2D_{\xi}^{1/2}\bigl(\chi F_{\mu}(t,\xi)\widehat{\phi}(\xi)\bigr) \notag \\
&\;\;\;+4i\mu t^2D_{\xi}^{1/2}\bigl(\chi \xi F_{\mu}(t,\xi)\widehat{\phi}(\xi)\bigr) -(6\mu^2+4)t^2 D_{\xi}^{1/2}\bigl(\chi \xi^2F_{\mu}(t,\xi)\widehat{\phi}(\xi)\bigr)-12i\mu t^2D_{\xi}^{1/2}\bigl(\chi \xi^3F_{\mu}(t,\xi)\widehat{\phi}(\xi)\bigr)  \notag \\
&\;\;\; +9\mu^2 t^2 D_{\xi}^{1/2}\bigl(\chi \xi^4\,F_{\mu}(t,\xi)\,\widehat{\phi}(\xi)\bigr)+2\mu t D_{\xi}^{1/2}\bigl(\chi \text{sgn}(\xi)F_{\mu}(t,\xi)\partial_{\xi}\widehat{\phi}(\xi)\bigr) + 4itD_{\xi}^{1/2}\bigl(\chi |\xi|\,F_{\mu}(t,\xi)\,\partial_{\xi}\widehat{\phi}(\xi)\bigr) \notag \\
&\;\;\; -6\mu t D_{\xi}^{1/2}\bigl(\chi \xi |\xi| \, F_{\mu}(t,\xi)\,\partial_{\xi}\widehat{\phi}(\xi)\bigr)+D_{\xi}^{1/2}\bigl(\chi F_{\mu}(t,\xi)\,\partial_{\xi}^2\widehat{\phi}(\xi)\bigr) \notag \\
&=K_1+K_2+K_3+K_4+K_5+K_6+K_7+K_8+K_9+K_{10}+K_{11}. \label{estimatek}
\end{align}
We show in detail the estimates for $K_7$ and $K_{11}$ which are the terms involving the highest regularity and decay of the initial data. Estimates for all the another terms in (\ref{estimatek}), except $K_8$, are obtained in a similar manner as $K_7$ and $K_{11}$. $K_8$ will be canceled with a term arising in the integral part in Duhamel's formula. 
For $K_7$ we obtain from Theorem \ref{derivaStein}, (\ref{dercuatro}), (\ref{unoa}) and fractional product rule type estimate (\ref{productostein}) that
\begin{align}
\nor{K_7}{}&\lesssim t^2\nor{ \chi \xi^4 F_{\mu}(t,\xi)\widehat{\phi}(\xi)}{} + t^2\nor{ \mathcal{D}_{\xi}^{1/2}\bigl(\chi \xi^4\,F_{\mu}(t,\xi)\,\widehat{\phi}(\xi)\bigr)}{} \notag \\
&\lesssim t^2\nor{ \chi \xi^4 F_{\mu}(t,\xi)\widehat{\phi}(\xi)}{} + t^{2/3}\Bigl(\nor{\chi \widehat{\phi}(\xi)}{}+\nor{|\xi|^{1/2}\chi \widehat{\phi}(\xi)}{}+\nor{ \mathcal{D}_{\xi}^{1/2}\bigl(\chi \widehat{\phi}(\xi)\bigr)}{}\Bigr) \notag \\
&\lesssim t^2e^{\mu t}\nor{\chi \xi^4}{\infty}\nor{\phi}{}+t^{2/3}\Bigl(\nor{\phi}{}+\nor{|\xi|^{1/2}\chi}{\infty} \nor{\phi}{}+\nor{ \mathcal{D}_{\xi}^{1/2}\bigl(\chi \bigr)}{\infty}\nor{\phi}{}+\nor{\chi}{\infty}\nor{\mathcal{D}_{\xi}^{1/2}\bigl(\widehat{\phi}(\xi)\bigr)}{}\Bigr) \notag \\
& \lesssim c(T) \nor{\langle x \rangle^{1/2} \phi }{}. \label{estimatek7}
\end{align}
and similarly
\begin{align}
\nor{K_{11}}{}&\lesssim \nor{ \chi  F_{\mu}(t,\xi)\partial_{\xi}^2 \widehat{\phi}(\xi)}{} + \nor{ \mathcal{D}_{\xi}^{1/2}\bigl(\chi F_{\mu}(t,\xi) \partial_{\xi}^2 \widehat{\phi}(\xi)\bigr)}{} \notag \\
&\lesssim \nor{ \chi }{\infty} e^{\mu t}\nor{\partial_{\xi}^2 \widehat{\phi}(\xi)}{} +\nor{\chi \partial_{\xi}^2\widehat{\phi}(\xi)}{}+\nor{|\xi|^{1/2}\chi \partial_{\xi}^2\widehat{\phi}(\xi)}{}+\nor{ \mathcal{D}_{\xi}^{1/2}\bigl(\chi \partial_{\xi}^2\widehat{\phi}(\xi)\bigr)}{} \notag \\
&\lesssim e^{\mu t}\nor{x^2 \phi}{}+\nor{\chi}{\infty}\nor{x^2\phi}{}+\nor{|\xi|^{1/2}\chi}{\infty} \nor{x^2\phi}{}+\nor{ \mathcal{D}_{\xi}^{1/2}\bigl(\chi \bigr)}{\infty}\nor{\partial_{\xi}^2\widehat{\phi}(\xi)}{}+\nor{\chi}{\infty}\nor{\mathcal{D}_{\xi}^{1/2}\partial_{\xi}^2 \bigl(\widehat{\phi}(\xi)\bigr)}{} \notag \\
& \lesssim c(T) \nor{\langle x \rangle^{2+1/2} \phi }{}. \label{estimatek11}
\end{align}
Now, let us go over the integral part in (\ref{d5/2fd}) that can be written in Fourier space and with the help of a commutator as
\begin{align}
-\int_0^t \chi D_{\xi}^{1/2}E(t-t',\xi,\widehat{z}(t',\xi))\,dt'&=\int_0^t[D_{\xi}^{1/2}, \chi ] E(t-t',\xi,\widehat{z}(t',\xi))\,dt' - \int_0^tD_{\xi}^{1/2}\bigl(\chi E(t-t',\xi,\widehat{z}(t',\xi))\bigr)\,dt' \notag \\
&=\mathcal{J}+\mathcal{K}. \label{jjmaskk}
\end{align}
where
\begin{align}
\mathcal{J}&= \int_0^t[D_{\xi}^{1/2}, \chi ] \Bigl( 2\mu (t-t')\delta \, \widehat{z}(t',\xi) +\Bigl( 2i(t-t')\,\text{sgn}(\xi)-6\mu (t-t') |\xi| +(t-t')^2\mu^2+4i\mu (t-t')^2\xi + \notag \\
&\qquad -(6\mu^2+4)(t-t')^2 \xi^2-12i\mu (t-t')^2 \xi^3+9\mu^2 (t-t')^2 \xi^4\Bigr) F_{\mu}(t-t')\widehat{z}(t',\xi) + \notag \\
&\qquad +2\mu (t-t') \,\text{sgn}(\xi)F_{\mu}(t-t',\xi)\partial_{\xi}\widehat{z}(t',\xi) + 4i(t-t') |\xi|\,F_{\mu}(t-t',\xi)\,\partial_{\xi}\widehat{z}(t',\xi)+ \notag \\
&\qquad -6\mu (t-t') \xi |\xi| \, F_{\mu}(t-t',\xi)\,\partial_{\xi}\widehat{z}(t',\xi)+F_{\mu}(t-t',\xi)\,\partial_{\xi}^2\widehat{z}(t',\xi) \Bigr) \,dt'\notag \\
&= \mathcal{J}_1+\mathcal{J}_2+\mathcal{J}_3+\mathcal{J}_4+\mathcal{J}_5+\mathcal{J}_6+\mathcal{J}_7+\mathcal{J}_8+\mathcal{J}_9+\mathcal{J}_{10}+\mathcal{J}_{11}+\mathcal{J}_{12}, \label{estimatejs}
\end{align}
and
\begin{align}
\mathcal{K}&=- \int_0^tD_{\xi}^{1/2}\Bigl( 2\mu (t-t')\chi \delta \, \widehat{z}(t',\xi) +\Bigl( 2i(t-t') \chi \,\text{sgn}(\xi)-6\mu (t-t')\chi |\xi| +(t-t')^2\mu^2\chi +4i\mu (t-t')^2 \chi \xi + \notag \\
&\qquad -(6\mu^2+4)(t-t')^2\chi \xi^2-12i\mu (t-t')^2\chi \xi^3+9\mu^2 (t-t')^2\chi \xi^4\Bigr) F_{\mu}(t-t')\widehat{z}(t',\xi) + \notag \\
&\qquad +2\mu (t-t') \,\text{sgn}(\xi)\chi F_{\mu}(t-t',\xi)\partial_{\xi}\widehat{z}(t',\xi) + 4i(t-t')\chi |\xi|\,F_{\mu}(t-t',\xi)\,\partial_{\xi}\widehat{z}(t',\xi)+ \notag \\
&\qquad -6\mu (t-t')\chi \xi |\xi| \, F_{\mu}(t-t',\xi)\,\partial_{\xi}\widehat{z}(t',\xi)+\chi F_{\mu}(t-t',\xi)\,\partial_{\xi}^2\widehat{z}(t',\xi) \Bigr) \,dt' \notag \\
&= \mathcal{K}_1+\mathcal{K}_2+\mathcal{K}_3+\mathcal{K}_4+\mathcal{K}_5+\mathcal{K}_6+\mathcal{K}_7+\mathcal{K}_8+\mathcal{K}_9+\mathcal{K}_{10}+\mathcal{K}_{11}+\mathcal{K}_{12}. \label{estimateks}
\end{align}
Notice that $\mathcal{J}_1$ and $\mathcal{K}_1$ vanish since $u\partial_xu$ has zero mean value and for $\mathcal{J}_2$, $\mathcal{J}_3$, $\mathcal{J}_4$, $\mathcal{J}_5$, $\mathcal{J}_6$, $\mathcal{J}_7$, $\mathcal{J}_8$, $\mathcal{J}_9$, $\mathcal{J}_{10}$, $\mathcal{J}_{11}$,  $\mathcal{J}_{12}$, $\mathcal{K}_2$, $\mathcal{K}_3$, $\mathcal{K}_4$, $\mathcal{K}_5$, $\mathcal{K}_6$, $\mathcal{K}_7$, $\mathcal{K}_8$, $\mathcal{K}_{10}$, $\mathcal{K}_{11}$ and $\mathcal{K}_{12}$, in $L^2(\mathbb{R})$ are essentially the same for their counterparts in equations (\ref{chiA}) and (\ref{chiB}), in the proof of Theorem \ref{contunica1}, so we omit the details of their estimates.
Therefore, from the assumption that $\phi=u(0)=u(t_1)$, $u(t_2) \in \dot{Z}_{5/2,5/2}$, equations (\ref{estimatejs}), (\ref{estimateks}) and the estimates above, we conclude that
\begin{align}
R&=K_8+\mathcal{K}_9 \notag \\
&=2\mu t D_{\xi}^{1/2}\bigl(\chi \text{sgn}(\xi)F_{\mu}(t,\xi)\partial_{\xi}\widehat{\phi}(\xi)\bigr)- \int_0^tD_{\xi}^{1/2}\Bigl(2\mu (t-t') \,\text{sgn}(\xi)\chi F_{\mu}(t-t',\xi)\partial_{\xi}\widehat{z}(t',\xi) \Bigr) \,dt', \label{R}
\end{align}
is a function in $L^2(\mathbb{R})$ at time $t=t_2$. But
\begin{align}
R&=2\mu t D_{\xi}^{1/2}\bigl(\chi \text{sgn}(\xi)F_{\mu}(t,\xi)\bigl(\partial_{\xi}\widehat{\phi}(\xi)-\partial_{\xi}\widehat{\phi}(0)\bigr)\bigr) \notag \\
&\;\;\;- 2\mu \int_0^tD_{\xi}^{1/2}\Bigl( (t-t') \,\text{sgn}(\xi)\chi F_{\mu}(t-t',\xi)\Bigl( \partial_{\xi}\Bigl(i\frac{\xi}{2}\widehat{u}*\widehat{u}(t',\xi)\Bigr)- \partial_{\xi}\Bigl(i\frac{\xi}{2}\widehat{u}*\widehat{u}(t',0)\Bigr)\Bigr)\Bigr) \,dt' \notag \\
&\;\;\;+2\mu t D_{\xi}^{1/2}\bigl(\chi \text{sgn}(\xi)F_{\mu}(t,\xi)\partial_{\xi}\widehat{\phi}(0)\bigr) \notag\\
&\;\;\;-2\mu \int_0^t(t-t') D_{\xi}^{1/2}\Bigl( \text{sgn}(\xi)\chi F_{\mu}(t-t',\xi)\partial_{\xi}\Bigl(i\frac{\xi}{2}\widehat{u}*\widehat{u}(t',0)\Bigr) \Bigr) \,dt' \notag \\
&=R_1+R_2+R_3+R_4. \label{IdR}
\end{align}
We argue like at the end of the proof of Theorem 3 in \cite{FP}. In this way, $R_1$ and $R_2$ are in $L^2(\mathbb{R})$ and this implies that $(R_3+R_4)(t_2)\in L^2(\mathbb{R})$. Also,
$$\partial_{\xi}\Bigl(i\frac{\xi}{2}\widehat{u}*\widehat{u}\Bigr)(0)=-i\int x u\partial_xu\,dx = \frac{i}{2}\nora{u}{}{2},$$
and from the npBO equation we have
\begin{equation}\label{intxnpbo}
\dfrac{d}{dt}\int xu\,dx+\int x\partial_x^2\mathcal{H}u\,dx+\int x u\partial_xu\,dx+\mu \int x \partial_x\mathcal{H}u\,dx+\mu \int x \partial_x^3\mathcal{H}u\,dx = 0,
\end{equation}
which shows that
\begin{equation}\label{intxnpbo1}
\dfrac{d}{dt}\int xu\,dx-\mu \int \mathcal{H}u\,dx = -\int x u\partial_xu\,dx=\dfrac{1}{2}\nora{u}{}{2},
\end{equation}
and hence
$$\partial_{\xi}\Bigl(i\frac{\xi}{2}\widehat{u}*\widehat{u}\Bigr)(0)=i\,\dfrac{d}{dt}\int xu\,dx \, ,$$
because $\widehat{\mathcal{H}u}(0)=\int \mathcal{H}u\,dx=0$. Now, substituting this into $R_4$ we have after integration by parts that
\begin{align}
R_4&=-2 i \mu D_{\xi}^{1/2}\left[ \text{sgn}(\xi)\chi \int_0^t(t-t') F_{\mu}(t-t',\xi)\Bigl( \dfrac{d}{dt}\int xu\,dx  \Bigr) \,dt' \right] \notag \\
&=-2 i \mu D_{\xi}^{1/2} \left[\text{sgn}(\xi)\chi (t-t') F_{\mu}(t-t',\xi)\left. \Bigl(\int xu\,dx \Bigr)  \right|_{t'=0}^{t'=t}\right.  +\notag \\
&\;\;\;+\left.\text{sgn}(\xi) \chi \int_0^t F_{\mu}(t-t',\xi) \Bigl(\int xu\,dx \Bigr)\,dt' + \text{sgn}(\xi)\chi \int_0^t (t-t')F_{\mu}(t-t',\xi)b_{\mu}(\xi) \Bigl(\int xu\,dx \Bigr)\,dt'  \right]  \notag \\
&=S_1+S_2+S_3. \label{s3}
\end{align}
Since $\partial_{\xi}\widehat{\phi}(0)=-i\widehat{x\phi}(0)=-i\int x\phi(x)\,dx$, then $S_1=-R_3$. We observe that $S_3$ in (\ref{s3}) belongs to $L^2(\mathbb{R})$, therefore 
$$S_2=-2 i \mu D_{\xi}^{1/2}\left(\text{sgn}(\xi) \chi \int_0^t  F_{\mu}(t-t',\xi) \Bigl(\int xu(x,t')\,dx \Bigr)\,dt'\right)$$
is in $L^2(\mathbb{R})$ at time $t=t_2$, and from Theorem \ref{derivaStein} this is equivalent to have that 
\begin{equation}\label{s2}
\mathcal{D}_{\xi}^{1/2}\left(\text{sgn}(\xi) \chi(\xi) \int_0^{t_2}  F_{\mu}(t_2-t',\xi) \Bigl(\int xu(x,t')\,dx \Bigr)\,dt'\right)\in L^2(\mathbb{R}),
\end{equation}
which from Proposition 3 in \cite{FP} implies that $\int_0^{t_2}\bigl(\int x u(x,t')\,dx\bigr)\,dt'=0$ and hence $\int xu(x,t)\,dx$ must be zero at some time in $(0,t_2)$. We re-apply the same argument to conclude that $\int xu(x,t)\,dx$ is again zero at some other time in $(t_2, t_3)$. Finally,  identity (\ref{intxnpbo1}) and the Fundamental Theorem of Calculus complete the proof of the theorem\,$\square$.


\renewcommand{\refname}

\end{document}